\documentclass[letterpaper,notitlepage,11pt,leqno]{article}

\title{Equivariant Framed Little Disk Operads are Additive}
\author{Ben Szczesny}
\date{}

\usepackage[margin=1.15in]{geometry} %
\usepackage[english]{babel}
\usepackage[utf8]{inputenc}
\usepackage{lmodern} %
\usepackage{amsmath, amssymb, amsthm, mathtools,thm-restate} %
\usepackage[bb=dsserif,bbscaled=1.1,scr=euler,cal=dutchcal]{mathalpha}
\usepackage[x11names]{xcolor}
\usepackage[pdfencoding=auto, psdextra, colorlinks, citecolor=MediumPurple3, linkcolor=MediumPurple3]{hyperref} %

\usepackage{bookmark}%
\usepackage[nameinlink,noabbrev]{cleveref}

\usepackage{setspace} 
\usepackage[inline]{enumitem} %
\setlist[enumerate,1]{nosep, label = (\arabic*), labelindent=\parindent}
\newlist{conditions}{enumerate}{1} %
\setlist[conditions]{label=(\arabic*),ref = (\theequation.\arabic*)}
\crefname{conditionsi}{condition}{conditions}

\usepackage{fancyhdr} 
\fancypagestyle{draft}{%
\fancyhf{}%
\fancyhead[R]{Compiled: \today}
\fancyfoot[C]{\thepage}%
}
\usepackage{thmtools,thm-restate}
\swapnumbers
\numberwithin{equation}{section}
\numberwithin{figure}{section}
\declaretheorem[name=Theorem,
refname={theorem,theorems},
Refname={Theorem,Theorems},
parent=section,
sibling=equation]{theorem}

\declaretheorem[name=Lemma,
refname={lemma,lemmas},
Refname={Lemma,Lemmas},
parent=section,
numberwithin=section,
sibling=equation]{lemma}

\declaretheorem[name=Corollary,
refname={corollary,corollaries},
Refname={Corollary,Corollaries},
parent=section,
sibling=equation]{corollary}

\declaretheorem[name=Definition,
refname={definition,definitions},
Refname={Definition,Definitions},
parent=section,
style=definition,
sibling=equation]{definition}

\declaretheorem[name=Convention,
refname={convention,conventions},
Refname={Convention,Conventions},
parent=section,
style=definition,
sibling=equation]{convention}

\declaretheorem[name=Example,
refname={example,examples},
Refname={Example,Examples},
parent=section,
style=definition,
sibling=equation]{example}

\declaretheorem[name=Remark,
refname={remark,remarks},
Refname={Remark,Remarks},
parent=section,
style=remark,
sibling=equation]{remark}

\usepackage{float}
\usepackage{graphicx} 
\usepackage{caption}
\usepackage{subcaption}
\usepackage{tikz}
\usetikzlibrary{fit}
\usetikzlibrary{positioning}
\usetikzlibrary{matrix}
\usetikzlibrary{calc}
\usetikzlibrary{intersections}
\usetikzlibrary{math}
\usetikzlibrary{shapes.geometric, arrows.meta}
\usepackage{quiver}
\tikzset{pullback/.style={minimum size=1.2ex,path picture={
\draw[opacity=1,black,-,#1] (-0.5ex,-0.5ex) -- (0.5ex,-0.5ex) -- (0.5ex,0.5ex);%
}}}

\usepackage[backend=biber,style=alphabetic,giveninits]{biblatex}
\usepackage{csquotes} %
\newcommand\restr[2]{{%
  \left.\kern-\nulldelimiterspace %
  #1 %
  \littletaller %
  \right|_{#2} %
  }}
\newcommand{\littletaller}{\mathchoice{\vphantom{\big|}}{}{}{}}

\DeclarePairedDelimiterX\set[1]\lbrace\rbrace{\def\given{\;\delimsize\vert\;}#1}%
\DeclarePairedDelimiter\abs{\lvert}{\rvert}
\DeclarePairedDelimiter\pbr{\lparen}{\rparen}
\DeclarePairedDelimiter{\norm}{\lVert}{\rVert}

\renewcommand{\phi}{\varphi}

\DeclareMathOperator*{\colim}{colim}
\DeclareMathOperator{\im}{im}

\DeclareMathOperator{\pr}{pr}

\newcommand{\gtop}[1][G]{\operatorname{Top}^{#1}}

\newcommand{\gsym}[1][G]{\Sigma\operatorname{Seq}^{#1}}

\newcommand{\goperad}[1][G]{\operatorname{Oper}^{#1}}
\newcommand{\groperad}[1][G]{\operatorname{Oper}^{#1}_0}

\DeclareMathOperator{\GL}{GL}
\DeclareMathOperator{\GA}{GA}

\DeclareMathOperator{\orth}{O}
\DeclareMathOperator{\DL}{DL}
\DeclareMathOperator{\DA}{DA}
\DeclareMathOperator{\decomposition}{Dec}
\DeclareMathOperator{\subspace}{Sub}

\DeclareMathOperator{\ambientoperad}{\mathscr{A}}
\newcommand{\staroperad}{\mathscr{S}}

\DeclareMathOperator{\littlediskoperad}{\mathscr{D}}
\DeclareMathOperator{\separateddiskoperad}{\mathscr{K}}
\DeclareMathOperator{\additivecore}{\mathscr{K}}

\DeclareMathOperator{\vertexinput}{In}
\DeclareMathOperator{\treeinput}{In}
\newcommand{\supertreegroupoid}{\mathcal{T}_{\text{sup}}}
\newcommand{\propertreegroupoid}{\mathcal{T}_{\text{prop}}}

\DeclareMathOperator{\supergenerators}{F_{sup}}
\DeclareMathOperator{\propergenerators}{F_{prop}}

\DeclareMathOperator{\coregenerators}{F_{core}}
\newcommand{\coretreegroupoid}{\mathcal{T}_{\text{core}}}

\newcommand{\Hill}{Hill}
\newcommand{\Blumberg}{Blumberg}
\newcommand{\Barata}{Barata}
\newcommand{\Moerdijk}{Moerdijk}
\newcommand{\Brinkmeier}{Brinkmeier}
\newcommand{\Dunn}{Dunn}

\newcommand{\May}{May}
\newcommand{\Guillou}{Guillou}
\newcommand{\Merling}{Merling}
\newcommand{\Osorno}{Osorno}
\newcommand{\Batanin}{Batanin}
\newcommand{\Boardman}{Boardman}
\newcommand{\Vogt}{Vogt}

\usepackage{semantex}
\NewVariableClass\MyVar[output=\MyVar]
\NewObject\MyVar\va {a}
\NewObject\MyVar\vb {b}
\NewObject\MyVar\vc {c}
\NewObject\MyVar\vd {d}
\NewObject\MyVar\ve {e}
\NewObject\MyVar\vf {f}
\NewObject\MyVar\vg {g}
\NewObject\MyVar\vh {h}
\NewObject\MyVar\vi {i}
\NewObject\MyVar\vj {j}
\NewObject\MyVar\vk {k}
\NewObject\MyVar\vl {l}
\NewObject\MyVar\vm {m}
\NewObject\MyVar\vn {n}
\NewObject\MyVar\vo {o}
\NewObject\MyVar\vp {p}
\NewObject\MyVar\vq {q}
\NewObject\MyVar\vr {r}
\NewObject\MyVar\vs {s}
\NewObject\MyVar\vt {t}
\NewObject\MyVar\vu {u}
\NewObject\MyVar\vv {v}
\NewObject\MyVar\vw {w}
\NewObject\MyVar\vx {x}
\NewObject\MyVar\vy {y}
\NewObject\MyVar\vz {z}
\NewObject\MyVar\vA {A}
\NewObject\MyVar\vB {B}
\NewObject\MyVar\vC {C}
\NewObject\MyVar\vD {D}
\NewObject\MyVar\vE {E}
\NewObject\MyVar\vF {F}
\NewObject\MyVar\vG {G}
\NewObject\MyVar\vH {H}
\NewObject\MyVar\vI {I}
\NewObject\MyVar\vJ {J}
\NewObject\MyVar\vK {K}
\NewObject\MyVar\vL {L}
\NewObject\MyVar\vM {M}
\NewObject\MyVar\vN {N}
\NewObject\MyVar\vO {O}
\NewObject\MyVar\vP {P}
\NewObject\MyVar\vQ {Q}
\NewObject\MyVar\vR {R}
\NewObject\MyVar\vS {S}
\NewObject\MyVar\vT {T}
\NewObject\MyVar\vU {U}
\NewObject\MyVar\vV {V}
\NewObject\MyVar\vW {W}
\NewObject\MyVar\vX {X}
\NewObject\MyVar\vY {Y}
\NewObject\MyVar\vZ {Z}
\NewObject\MyVar\valpha {\alpha}
\NewObject\MyVar\vvaralpha {\varalpha}
\NewObject\MyVar\vbeta {\beta}
\NewObject\MyVar\vgamma {\gamma}
\NewObject\MyVar\vdelta {\delta}
\NewObject\MyVar\vepsilon {\epsilon}
\NewObject\MyVar\vvarepsilon {\varepsilon}
\NewObject\MyVar\vzeta {\zeta}
\NewObject\MyVar\veta {\eta}
\NewObject\MyVar\vtheta {\theta}
\NewObject\MyVar\viota {\iota}
\NewObject\MyVar\vkappa {\kappa}
\NewObject\MyVar\vlambda {\lambda}
\NewObject\MyVar\vmu{\mu}
\NewObject\MyVar\vnu{\nu}
\NewObject\MyVar\vxi{\xi}
\NewObject\MyVar\vpi{\pi}
\NewObject\MyVar\vvarpi {\varpi}
\NewObject\MyVar\vrho {\rho}
\NewObject\MyVar\vsigma {\sigma}
\NewObject\MyVar\vtau {\tau}
\NewObject\MyVar\vupsilon {\upsilon}
\NewObject\MyVar\vphi {\phi}
\NewObject\MyVar\vvarphi {\varphi}
\NewObject\MyVar\vchi {\chi}
\NewObject\MyVar\vpsi {\psi}
\NewObject\MyVar\vomega {\omega}
\NewObject\MyVar\vGamma {\Gamma}
\NewObject\MyVar\vDelta {\Delta}
\NewObject\MyVar\vTheta {\Theta}
\NewObject\MyVar\vLambda {\Lambda}
\NewObject\MyVar\vXi{\Xi}
\NewObject\MyVar\vPi{\Pi}
\NewObject\MyVar\vSigma {\Sigma}
\NewObject\MyVar\vUpsilon {\Upsilon}
\NewObject\MyVar\vPhi {\Phi}
\NewObject\MyVar\vPsi {\Psi}
\NewObject\MyVar\vOmega {\Omega}

\NewObject\MyVar\id{\mathbb{1}}

\NewObject\MyVar\Orth{\operatorname{O}}

\NewObject\MyVar\fOrd{\operatorname{\textbf{fOrd}}}

\NewObject\MyVar\Aut{\operatorname{Aut}}
\NewObject\MyVar\Hom{\operatorname{Hom}}
\NewObject\MyVar\Rad{\operatorname{Rad}}
\NewObject\MyVar\ar{\operatorname{ar}}
\DeclareMathOperator{\rad}{Rad}
\newcommand{\Ord}[1]{\underline{#1}}

\NewObject\MyVar\LittleCubeOperad{\mathcal{C}}

\SetupClass\MyVar{
  define keys = {
  {cal}{ command = \mathcal},
  {scr}{ command = \mathscr},
  {~}{command = \widetilde},
  {oline}{command = \overline},
  {'}{upper=\prime},
  {inv}{upper=-1},
  {^}{command = \widehat{}},
  },
  define keys[1] = {
    {^}{ upper = {#1}}, 
  },
}

\begin{document}

\maketitle
\begin{abstract}
    In this paper, we generalize the \Dunn-\Brinkmeier~additivity theorem, which establishes a weak equivalence $\mathcal{C}_n \otimes \mathcal{C}_m \simeq \mathcal{C}_{n+m}$ for the little cubes operad $\mathcal{C}_n$. We introduce equivariant framed little disk operads, a new class of operads that simultaneously generalize the framed little disk operads and the little $V$-disk operads associated with a $G$-representation $V$. We prove that these operads satisfy an analogous additivity property, extending the classical theorem to settings involving group actions and framings. 
\end{abstract}
\tableofcontents
\section{Introduction}\label{sec:introduction}
In this paper we prove a generalization of \Dunn~and \Brinkmeier's classic result \cite{dunnTensorProductOperads1988, brinkmeierOperads2000} that the little cube operads are additive.
That is, there is a weak equivalence of operads \[\mathcal{C}_n\otimes \mathcal{C}_m\simeq \mathcal{C}_{n+m}\] where $\mathcal{C}_n$ is the operad of little $n$-cubes and $\otimes$ is the Boardman-Vogt tensor product of operads.

The main contribution of this paper is that we prove that, instead of cubes, the little disk operads $\mathcal{D}_n$ are additive.
To those less familiar with the tensor on operads, this may seem like a trivial change.
However, the classic tensor on operads \emph{isn't homotopical}, and the proofs of the \Dunn-\Brinkmeier~additivity theorem make use of the geometry of rectangles in a fundamental way.
In order to prove the additivity theorem for disks then requires new ideas to get around the obstructions this seemingly simple change makes. In particular, in this paper we will introduce a geometric way to compare embedding operads of different shapes, as well as develop more general conditions for when induced operad maps from the tensor $P\otimes Q\to R$ are injective.

While cubes are traditionally the preferred operad to encode $\mathbb{E}_k$-algebra structures, they tend to be ill-suited for various equivariant generalizations.
Instead, little disks are much more preferable and this work was initiated by the authors interest in these generalizations.

If we take the little disks to live in a $G$-representation $V$, then we get the $G$-operad of equivariant little $V$-disks $\mathcal{D}_V$.
These are studied, among others, by \May-\Guillou~\cite{guillouEquivariantIteratedLoop2017} and \May-\Merling-\Osorno~\cite{mayEquivariantInfiniteLoop2017} in relation to equivariant loop space theory.
When $V$ is a $G$-universe, then $\mathcal{D}_V$ are examples of what \Blumberg-\Hill~\cite{blumbergOperadicMultiplicationsEquivariant2015} call $\mathbb{N}_\infty$-operads. Blumberg-Hill's interest in these kinds of operads stems from the need that, equivariantly, these kinds of operads not only record homotopically coherent commutative algebra structures, but also a system of norm maps. One can think of these norm maps as a kind of ``twisted'' multiplication.

Besides encoding higher equivariant algebra structures, little disks also come up in other areas such as factorization homology, except this time as \emph{framed} little disks $\mathcal{D}_n^{f}$.
Disks are much more preferred in this context, even though it is possible to amend little cubes to also carry framing data as done by Dwyer-Hess-Knudsen~\cite{dwyerConfigurationSpacesProducts2018}. 
The motivation for coming up with this amendment is to use a conjectured additivity theorem \cite[Conjecture 4.18]{dwyerConfigurationSpacesProducts2018}. While we don't prove this conjecture exactly, we do prove one that is in spirit the same. We haven't verified it, but the results of our paper should be a sufficient replacement for their applications.

It turns out that extending a non-equivariant additivity result on disks to both of these generalizations is not that much harder to deal with, and we define \emph{equivariant framed little disk operads} $\mathcal{D}^\rho(V)$ where $\rho$ is a dilation representation (defined later) which carries the framing data, and $V$ an orthogonal $G$-representation of countable dimension.
The main result of this paper is the following.

\begin{theorem}\label{thm:maindiskadditiviy}
    Given $G$-representations $V$ and $W$, and dilation representations $\rho: \mathcal{G}\to \GL(V)$, and $\psi: \mathcal{H}\to \GL(W)$, there is a weak equivalence of $G$-operads \[\mathcal{D}^\rho_V\otimes \mathcal{D}^\psi_W \simeq \mathcal{D}^{\rho\times \psi}_{V\oplus W}.
    \]
\end{theorem}

To better put our proof into context, it is worthwhile to quickly review how the proof of the \Dunn-\Brinkmeier~Additivity Theorem works.
There are obvious operad maps \[\mathcal{C}_n\xrightarrow{i_1} \mathcal{C}_{n+m} \xleftarrow{i_2} \mathcal{C}_m\]
which amount to extending cubes by mapping a cube to the product of the same cube and ``identity'' cube.
These maps interchange, and we get an induced map of operads \[\mathcal{C}_n\otimes \mathcal{C}_m\to \mathcal{C}_{n+m}.
\] It is not too difficult to show that codomain deformation retracts onto the image of this map, and the hard part of the proof is then showing that this map is an embedding.

When generalizing to little disk operads $\mathcal{D}_n$, the first issue we run into is that products of disks aren't disks.
This means we don't have a straight forward induced map like \[\mathcal{D}_V^\rho\otimes \mathcal{D}_W^\psi \to \mathcal{D}_{V\oplus W}^{\rho\times\psi}.
\] This forces us to consider more general operads than just the equivariant framed little disks $\mathcal{D}_V^\rho$, as we will need to deal with embeddings of products of disks.
Instead, in \Cref{sec:framedlittledisks}, we will define the framed equivariant little star operads $\staroperad^\rho_V(S)$, where $S$ is some open star domain centered at the origin.
Writing $B_V$ for the open unit disk of $V$ centered at the origin, we recover the little disk operads by $\mathcal{D}_V^\rho=\staroperad^\rho(B_V)$.

We will only consider star shape domains that are products of disks in this paper, but much of what we will say still holds for more general star shapes.
We have decided to limit ourselves to disks in order to make the exposition clearer - dealing with arbitrary shapes involves considering many more conditions and ad-hoc arguments to pull it together.
Disks are our desired shape anyway.

The candidate induced map for an additivity theorem is then of the form
\begin{align}
    \mathcal{D}_V^\rho\otimes \mathcal{D}^\psi_W\to\staroperad^{\rho\times \psi}(B_V\times B_W). \label{eq:candidatemap}    
\end{align}
We will show in this paper $\staroperad^{\rho\times \psi}(B_V\times B_W)$ is weakly equivalent to $\mathcal{D}_{V\oplus W}^{\rho\times \psi}$, and that the map of \Cref{eq:candidatemap} is a weak equivalence.

\subsection{Outline of the Paper}
Let us now outline the structure of this paper and the main steps of our proof.  

\Cref{sec:conventions} is primarily devoted to setting conventions used throughout the paper. In the second part we define the notion of an \emph{algebraically axial operad} which are operads determined by its ``components''. This will be a key property that we will later use to prove when induced maps from the tensor is injective.

We then move on in \Cref{sec:tensor} to quickly remind the reader about the Boardman-Vogt tensor and its main properties. We then introduce the notion of \emph{trees in superposition} to help pick out specific elements of the tensor. This material will be useful for certain constructions we use in the main proof.

We define equivariant framed little disks $\littlediskoperad^\rho_V$ and the related framed little star operads $\staroperad^\rho(S)$ in \Cref{sec:framedlittledisks}, before moving on in \Cref{sec:comparison} where we prove equivalences between these operads which generalize the well known fact that little cube operads and little disk operads are equivalent.

The proof technique that we employ in proving injectivity statements from the operad tensor was greatly influenced by a recent preprint of \Barata-\Moerdijk~\cite{barataAdditivityLittleCubes2022} which streamlined \Brinkmeier's argument \cite{brinkmeierOperads2000}. In \Cref{sec:divisibility} we define some key properties and illustrate how \Barata-\Moerdijk's argument works.

We finally arrive to the main theorem of this paper in \Cref{sec:additivity}, which is devoted to showing that the candidate map of  \Cref{eq:candidatemap} is a weak equivalence.
Unlike cubes, this map \emph{isn't injective.
}
Instead, we will identify a subcollection $\additivecore \subseteq \mathcal{D}_V^\rho\otimes \mathcal{D}^\psi_W$ which we will call the \emph{additive core}.
We will show that the induced map restricted onto $\additivecore$ is an embedding.
Once we establish this, we will show that both the domain and codomain are weakly equivalent to the additive core $\additivecore$ which then implies \Cref{thm:maindiskadditiviy}.

We also include two appendices in this paper where we put required ``elementary'' topology results. The first appendix section \Cref{appendix: correction lemmas} gives two lemmas on when maps of the form $X\times [0,\infty)\to X$ can be converted to explicit deformations $X\times [0,1]\to X$ onto a subspace. The second section \Cref{appendix:topology lemmas} we go through some needed verifications to confirm that certain injections are in fact embeddings. 

\subsection{Acknowledgments}

An earlier incomplete proof of our main result was part of the Author's thesis. We thank Mike Hill for his guidance and advice during the time that the main ideas were formed.

\section{Conventions}\label{sec:conventions}

\subsection{General Conventions}

Throughout this paper we will be mixing different kinds of groups together.
To differentiate these two cases better, we will follow the convention that \emph{finite groups} will be denoted by capital letters:
$\vG,\vH,\vK,\dots$, while a topological group will be denoted by calligraphic scripts:
$\vG[cal], \vH[cal], \vK[cal],\dots$. In fact, we will be dealing with $G$-topological groups. That is, group objects in $\gtop$, the category of topological spaces with $G$-action and equivariant maps.
We will assume throughout that topological groups are Hausdorff and paracompact. We will assume that a $G$-representation $V$ is a countable dimension orthogonal $G$-representation, and a direct sum of orthogonal representations is an orthogonal direct sum. i.e., If $\vV$ and $\vW$ are $\vG$-representations then $V\oplus W$ is such that $W$ and $W$ are orthogonal to each other.

\begin{convention}
    Since we only deal with countable dimensional vector spaces we are going to slightly restrict what we consider $\Orth{\vV}$ to be in the case when $\dim(V)=\infty$.
    In particular, given any finite subspace $W \subseteq V$, we define $\Orth{\vW}$ as usual which is topologized as a subspace of $\mathbb{R}^{\dim(W)^2}$.
    For $W \subseteq W^{\prime}$ we have the homomorphism \[\Orth{\vW} \hookrightarrow \Orth(\vW['])\] which extends maps by identity.
    We then define \[\Orth{V} = \colim_{W \subseteq V}\Orth(W).
    \]In particular, for us, all orthogonal maps are just identity outside some finite subspace.
\end{convention}

Besides linear transformations, we will also be using affine transformations. We denote by $\GA(V)$ the group of all general linear affine transformations on $V$. This has the structure of a $G$-group via conjugation. 

\subsection{Operads Indexed on Finite Ordinals}
We will be using the formulation of operads as functors on the category of finite ordinals satisfying extra conditions, following \Batanin~\cite{bataninEckmannHiltonArgumentHigher2008}. In particular, we denote the category of finite ordinals by $\fOrd$. This is the category spanned by the sets \[\Ord{n}=\set{1,2,\dots,n},\] including the object $\Ord{0}=\emptyset$, with morphisms given by any set map. The subcategory of $\fOrd$ of all objects with bijections as the morphisms will be denoted by $\Sigma$. A \emph{left symmetric $G$-sequence} is then a functor \[X: \Sigma \to \gtop\] and a morphism of left symmetric $G$-sequences is just a natural transformation of the functor. We will denote the category of left symmetric $G$-sequences by $\gsym$.

\begin{convention}
	Whenever we have a subset $A \subseteq \Ord{n}$, unless we say otherwise, we will put the induced order on the elements of $A$ and then identify it with the unique ordinal $\Ord{k}$ where $k=\abs{A}$. If we need to be explicit about the elements of $A$, we will denote the unique order preserving bijection by $\chi_A:\Ord{k}\to A$.
\end{convention}

Given a map $\alpha: \Ord{m}\to \Ord{n}$ and $i\in \Ord{n}$, we will denote the fiber over $i$ by $\Ord{m_i}$. i.e., the pullback
$$\begin{tikzcd}
	\Ord{m_i}\ar[r]\ar[d]\ar[rd,phantom,"\lrcorner" pos = 0.01] & \set{i}\ar[d] \\ \Ord{m}\ar[r,"\alpha"] & \Ord{n}
\end{tikzcd}$$

\begin{definition}
	A $G$-operad $P$ is a left symmetric $G$-sequence $P$ where we have a unit map \[\id: \ast \to P(1),\] and for each map $\valpha : \Ord{m}\to \Ord{n}$ we have a composition morphism \[\gamma_P(\alpha): P(n)\times P(m_1)\times P(m_2)\times \cdots \times P(m_n)\to P(m)\] that satisfy associative and unital conditions. A morphism of $G$-operads is a morphism of symmetric $G$-sequences which commutes with the operad structure. We will denote the category of $G$-operads by $\goperad$. A \emph{reduced} operad $P$ is one where $P(0)=\ast$.
	We will denote the subcategory of reduced $G$-operads by $\groperad$.
\end{definition}

Note that the composition maps $\gamma_\alpha$ must be defined for all morphisms of $\fOrd$ and not just order preserving maps. 

\begin{remark}{}{}
	We have used the convention of left actions for our operads, which is different from the usual right action convention.
	We do this to follow the convention in \cite{blumbergOperadicMultiplicationsEquivariant2015}.
	The advantage of this convention is that it puts the actions by $\Sigma$ and $G$ on the left, so we can think of each space $P(n)$ as a $\Sigma_n\times G$-space.
\end{remark}

We will often work with equations of operad elements.
In order to make our notation less cumbersome when doing so, we will now set up some notation to ease how we talk about them. In general, given an operad $P$, we will write $x\in P$ and call $x$ an \emph{element of $P$} when $x\in P(n)$ for some $n$, We will also set \[\ar(x):= \Ord{n},\]
and call $\ar(x)$ the \emph{indexing set or ordinal} of $x$. We will also say that an element $x\in P$ \emph{has $n$-arity}.
We will also use $P_{\leq k}$ and $P_{<k}$ to mean all elements with arity less than or equal $k$, and less than $k$ respectively.

Given a composition map \[\gamma_P(\alpha): P(n)\times \left(\bigtimes_{j\in n}P(m_j)\right)\to P(m),\] for $x\in P(n)$ and $y^j\in P(m_j)$ we will write the element-wise composition by \[x\underset{\alpha}{\circ} (y^j)_{j\in\ar(x)} := \gamma_P(\alpha)(x,(y^j)_{j\in\ar(x)}).\]
We will rarely write the map $\alpha$ in our composition notation as it usually implied by the order. Explicitly, suppose we have $x\in P(n)$ and $y^j\in P(m_j)$. We put the order on $\sqcup_{j\in \Ord{n}}\Ord{m_j}$ given by $a\leq b$ if either $a,b\in \Ord{m_j}$ and $a\leq b$ or if $a\in \Ord{m_j}$, $b\in\Ord{m_{\vj[']}}$ then $j\leq \vj[']$. There is then a unique order preserving map \[\beta:\Ord{m}\to \sqcup_{j\in \Ord{n}}\Ord{m_j} = \Ord{m_1+m_2+\dots + m_n}\] such that the fibers are $\Ord{m_j}$. The composition $x\circ (y^j)_{j\in\ar{x}}$ then refers to the composition with respect to the map $\beta$. While we omitted the associativity and unital conditions above, let us write down these conditions here in this notation for future reference.

Given $x,y^j,z^k\in P$, the associativity condition is 
\begin{align}
	x\circ \pbr{y^j\circ \pbr{z^k}_{k\in\ar{y^j}}}_{j\in\ar{x}}	= \pbr[\Big]{x\circ \pbr{y^j}_{j\in\ar{x}}}\circ \pbr{z^k}_{k\in \ar{x\circ (y^j)_{j\in\ar{x}}}}
\end{align} where $\ar{x\circ (y^j)_{j\in\ar{x}}}=\sqcup_{j\in\ar{x}}\ar{y^j}$. The unital conditions are given by 
\begin{align}
	x\circ (\id)_{i\in\ar{x}} = x = \id \circ x.
\end{align}
Lastly, the composition maps imply the usual equivariance conditions. Given $\sigma\in\vSigma[\ar{x}]$, and $\vtau_j\in\vSigma[\ar{y^j}]$, we will abuse notation and write $\sigma$ for $P(\sigma)$ and similarly for $\tau_j$ for $\vP{\tau_j}$. We write $\vsigma{\ar{y^j}}$ for the induced permutation on $\ar{x\circ (y^j)}$ from $\sigma$ and $\oplus_{i}\tau_j$ for the block sum permutation. The equivariance conditions can then be stated by 
\begin{align}
	(\sigma x)\circ (y^j)_{j\in\ar{x}} &= \vsigma{\ar{y^j}}\pbr[\Big]{x\circ (y^{\sigma{j}})_{j\in\ar{x}}} \\ 
	x\circ (\tau_jy^j)_{j\in\ar{x}} &= \pbr[\Big]{\oplus_{i}\tau_j}\pbr[\Big]{x\circ (y^j)_{j\in\ar{x}}}.
\end{align}

\subsection{Components of Reduced Operads}
Given a reduced operad $P$, element $x\in P$ and $i\in \ar(x)$, we have the ``components'' $x_i\in P(1)$ which are given by \[x_i := x\circ (\delta_i^j)_{j\in \ar(x)}\] where \[\delta^j_i := \begin{cases*}\ast & if $j\neq i$ \\ \id & if $j=i$.\end{cases*}\]

\begin{example}
	In the case of the little cube operads $\LittleCubeOperad[n]$, given $x=(c_1,c_2,\dots c_n)\in \LittleCubeOperad[n]$, the components of $x$ are exactly the components of the tuple $c_i$.
\end{example}

We will be interested in operads that whose elements are determined by their components. Our choice of name for the following definition is due to the similarity to \emph{axial operads} as defined in \cite{fiedorowiczAdditivityTheoremInterchange2015}

\begin{definition}\label{def:axial}
	An operad $P$ is \emph{algebraically axial} if it is reduced and for all $\Ord{n}\in \fOrd$, the component tuple maps \[P(n)\to \prod_{i\in I}P(\Ord{n})\] are injective.
\end{definition}

In particular, elements are determined by their components in such operads.
Generalizing this slightly, given $x\in P$ and $A \subseteq \ar(x)$, we will write \[\delta_A^j :
	= \begin{cases*}
		\ast & if $j \notin A$ \\ \id & if $j\in A$,
	\end{cases*}\] and then set $x_A := x\circ (\delta_A^j)_{j\in \ar(x)}
$.
The following lemma is straightforward.
\begin{lemma}{}{}
	Given an algebraically axial operad $P$, an element $x\in P$, and a partition $\set{U_i}$ of $\ar(x)$, then the map \[P(\ar(x))\to \prod_{i}P(U_i)\] given by $x \mapsto (x_{U_i})_i$ is injective.
	In particular, for $x,y$ such that $\ar(y)=\ar(x)$ we have that $x=y$ if and only if $x_{U_i}=y_{U_i}$ for all $i$.
\end{lemma}

An operad being algebraically axial is related to that operad being built from a monoid. To make this relationship precise, let us review a useful adjunction.

\begin{definition}
	Given a topological $G$-monoid $A$, we can construct a reduced $G$-operad $\mathscr{O}(A)$ as follows.
	The $n$-ary components are given by \[\mathscr{O}(A)(n)= \prod_{j\in \Ord{n}} A,\] and the composition is given by distributing the multiplication across the elements.
	That is, for a map of finite ordinals $\alpha:
		\Ord{m}\to \Ord{n}$, the composition is
	\begin{align*}
		\gamma(\alpha): \mathscr{O}(A)(n)\times  \prod_{j\in \Ord{n}} \mathscr{O}(A)(\Ord{m_j}) & \to \mathscr{O}(A)(\Ord{m}) \\ \left((a_j)_{j\in \Ord{n}},((b_k)_{k\in \Ord{m_j}})_{j\in \Ord{n}}\right)&\mapsto (a_{\alpha(k)}b_k)_{k\in \Ord{m}}.
	\end{align*}
\end{definition}
The following is due to Igusa~\cite{igusaAlgebraicKtheoryA_1982}; See also \cite[Lemma 7.2]{fiedorowiczAdditivityTheoremInterchange2015}.
\begin{theorem}{}{}
	The construction above is functorial and has a left adjoint given by taking the space indexed by $\Ord{1}$.
	\begin{equation*}
		\begin{tikzcd}
			\groperad & {\operatorname{Mon}(\gtop)}
			\arrow[""{name=0, anchor=center, inner sep=0}, "{\mathscr{O}}", curve={height=-18pt}, from=1-2, to=1-1]
			\arrow[""{name=1, anchor=center, inner sep=0}, "{(\cdot)(\Ord{1})}", curve={height=-18pt}, from=1-1, to=1-2]
			\arrow["\dashv"{anchor=center, rotate=-90}, draw=none, from=1, to=0]
		\end{tikzcd}
	\end{equation*}
\end{theorem}

Using this adjunction, another way to phrase \cref{def:axial} is that algebraically axial operads are precisely the operads $P$ such that the unit map of this adjunction on $P$
\[\mu_P : P\to \mathscr{O}(P(\Ord{1}))\] is injective.

\section{Boardman-Vogt Tensor and Trees in Superposition}\label{sec:tensor}
\subsection{The \Boardman-\Vogt~Tensor of Operads}
The \Boardman-\Vogt~tensor $P\otimes Q$ of operads $P$ and $Q$ is the universal construction that encodes ``interchanging'' $P$ and $Q$-algebra structures. 
We will assume here that the reader has some experience with the tensor, and use this section to primarily set notation and conventions. The original construction of \Boardman-\Vogt~can be found in \cite{boardmanHomotopyInvariantAlgebraic1973}, however this uses the language of PROPs. For a more modern treatment of the \Boardman-\Vogt~tensor, we refer the reader to \cite{heutsSimplicialDendroidalHomotopy2022}. The author's thesis~\cite{szczesnyAdditivityEquivariantOperads2023} also contains a more categorical flavored construction of the tensor.

For two ordinals $\Ord{n}$, and $\Ord{m}$, we will interpret the product $\Ord{m}\times \Ord{n}$ to be the ordinal $\Ord{mn}$ via the lexicographical ordering. The bijection that swaps factors of a product then corresponds to a permutation on $\Ord{mn}$ which we denote by $\tau_{m,n}$. i.e., the bijection that makes the following commute.
$$\begin{tikzcd}
	\Ord{m}\times \Ord{n} \ar[d,"\text{lex}"]\ar[r,"\text{swap}"] & \Ord{n}\times \Ord{m} \ar[d,"\text{lex}"]\\ 
	\Ord{mn}\ar[r,"\tau_{m,n}"] & \Ord{mn}.
\end{tikzcd}$$

Given two operads $P$ and $Q$, the tensor $P\otimes Q$ is then constructed by quotienting out the \emph{interchange relations} in the coproduct of operads $P\sqcup Q$. In particular, suppose we have two operads $P$ and $Q$, then in the coproduct $P\sqcup Q$ we have the following (non-commutative) diagram.

\[\begin{tikzcd}
	P(m)\times Q(n)\ar[r]\ar[d,"\text{swap}"] & \pbr{P\sqcup Q}(mn) \\
	Q(n)\times P(m)\ar[r] & \pbr{P\sqcup Q}(mn) \ar[u,"\tau_{n,m}"]
\end{tikzcd}\]
where the horizontal arrows are given by the compositions \[(p,q)\mapsto p\circ (q)_{i\in \ar{p}}.\]
i.e., for each $p\in P(m), q\in Q(n)$, we have two elements \[p\circ (q)_{i\in\ar{p}},\text{ and } \tau_{n,m}\pbr{q\circ (p)_{j\in\ar{q}}}\] in the coproduct $P\sqcup Q$. The interchange relations are then given by \[p\circ (q)_{i\in\ar{p}} \sim \tau_{n,m}\pbr{q\circ (p)_{j\in\ar{q}}}.\]
We will use the notation $p\otimes q$ to be the element in the tensor $P\otimes Q$ represented by $p\circ (q)_{i\in\ar(p)}$. We then have that \[p\otimes q = \tau_{n,m}(q\otimes p)\] in the tensor.

The Boardman-Vogt tensor satisfies a universal property.
Given operadic maps $f:P\to R$ and $g: Q\to R$, we say that the maps $f$ and $g$ \emph{interchange} if the following diagram commutes.

\[\begin{tikzcd}
	P(m)\times Q(n)\rar["\text{swap}"]\dar["\id\times \Delta"']& Q(n)\times P(m)\dar["\id\times \Delta"] \\ 
	P(m)\times \prod_{j\in \Ord{m}}Q(n)\dar["f\times (g)^m"']& Q(n)\times \prod_{k\in \Ord{n}}P(m)\dar["g\times (f)^n"] \\
	R(m)\times \prod_{j\in \Ord{m}}R(n)\dar["\gamma_R"']& R(n)\times \prod_{k\in \Ord{n}}R(m) \dar["\gamma_R"]\\ 
	R(mn)\rar["R(\tau_{m,n})"]& R(mn)
\end{tikzcd}\]

In which case, the universal property of the Boardman-Vogt tensor is the following.
\begin{theorem}
	Suppose $f:P\to R$ and $g:Q\to R$ are interchanging morphisms of $G$-operads. Then there exists unique $G$-operad morphisms $i_P,i_Q,$ and $f\otimes g$ that make the following diagram commute.
	\[\begin{tikzcd}
		P \rar["i_P"]\drar["f"']& P\otimes Q \dar["f\otimes g"]& Q\ar[l,"i_Q"']\ar[ld,"g"] \\
		&R&
	\end{tikzcd}\]
\end{theorem}

\subsection{Trees in Superposition}

The structure of the operad tensor can be very complicated.
The proof of our main theorem will require us to delve into this structure very explicitly, and we will need some constructions in order to more concretely talk about elements inside it. In an operad, compositions of elements are best talked about via tree diagrams where elements are associated to vertices of the tree and the edges tell us how they are composed together. We will extend this idea to tensors.

Our starting observation is that in the tensor $P\otimes Q$, we can write any element as iterated compositions of ``simple tensors'' $p\otimes q$ where $p\in P$ and $q\in Q$.
For instance, given $p\circ (q^i)_{i\in \ar(p)}$, this can be written as \[p\circ (q^i)_{i\in \ar(p)} = (p\otimes \id_Q)\circ (\id_P\otimes q^i)_{i\in\ar(p)}.\]

This suggests that we want to associate simple tensors to vertices. In order to do this, we will extend the idea of trees as follows. First, we will interpret for a tree $T$ that the set of vertices $V(T)$ and edges $E(T)$ are objects of $\fOrd$.

\begin{definition}{}{}
	A tree in superposition is the data $(T,\chi_W^T,\chi_B^T, \xi^T)$ where:
	\begin{enumerate}
		\item $T$ is a rooted tree.
		\item Viewing the vertex set $V(T)$ as a discrete category, we have the \emph{white and black incoming edge label functors}.
		      \begin{align*}
			      \chi_W:V(T) & \longrightarrow \fOrd \\
			      \chi_B:V(T) & \longrightarrow \fOrd
		      \end{align*}
		      For $v\in V(T)$, we call $\chi_W(v)$ the \emph{white incoming edge labels} and $\chi_B(v)$ the \emph{black incoming edge labels}.
		\item For $v\in\vV{T}$, we write $\vertexinput(v)$ for the set of incoming edges. There is a natural isomorphism \[\xi: \chi_W\times \chi_B \overset{\cong}{\Longrightarrow} \operatorname{In}:V(T)\to \fOrd\] which we call the \emph{incoming edge labelling transformation}.
		      In particular, for each $v\in V(T)$, we have a bijection \[\xi(v):
			      \chi_W(v)\times \chi_B(v)\to \operatorname{In}(v),\] and each incoming edge is then determined by a pair of a white and black label.
	\end{enumerate}

	An isomorphism of trees in superposition $(\alpha,\gamma_W,\gamma_B):
		T\to T^{\prime}$ is the data of:
	\begin{enumerate}
		\item a tree isomorphism $\alpha: T\to T^{\prime}$,
		\item natural isomorphisms \begin{align*}
			      \gamma_W & : \chi_W\Rightarrow \chi_W^{\prime}\circ \alpha \\
			      \gamma_B & : \chi_B\Rightarrow \chi_B^{\prime}\circ \alpha
		      \end{align*}
		      such that the following commutes
		      \[\begin{tikzcd}
				      {\chi_W\times \chi_B } & \vertexinput \\
				      {(\chi^\prime_W\circ\alpha)\times (\chi^\prime_B\circ\alpha)} & {\vertexinput^{\prime}\circ\alpha}
				      \arrow["\gamma_W\times \gamma_B"',Rightarrow, from=1-1, to=2-1]
				      \arrow["\xi^{\prime}",Rightarrow, from=2-1, to=2-2]
				      \arrow["\xi",Rightarrow, from=1-1, to=1-2]
				      \arrow[Rightarrow, from=1-2, to=2-2]
			      \end{tikzcd}\]
	\end{enumerate}

	We will denote the groupoid of trees in superposition by $\supertreegroupoid$. Note, that while we are interpreting our sets to be ordinals with an order, we are using any maps between them. In particular, we are not assuming that tree isomorphisms preserve this ordering on the vertices and edges. 

	For a finite ordinal $\Ord{n}$, we will denote the fiber of the tree input functor $\treeinput:
		\supertreegroupoid\to \Sigma$ over $\Ord{n}$ by $\supertreegroupoid(n)$.
\end{definition}

\subsection{Representatives in the Tensor}

A vertex $v$ such that $\abs{\vertexinput(v)}=0$ is called a stump. Given a non-stump vertex $v$ and incoming edge labelled by $(i,j)\in \chi_W(v)\times \chi_B(W)$, if the edge $(i,j)$ isn't an input edge for the tree $T$, we will denote the descendant vertex along this edge by $d_{(i,j)}(v)$.

\begin{definition}{}{}
	A tree in superposition $T$ is \emph{reduced} if for every non-stump vertex $v$, we have that for every $i\in\chi_W(v)$, there exists $j\in \chi_B(v)$ such that $d_{(i,j)}(v)$ isn't a stump. Similarly, for every $j\in\chi_B(v)$, there exists $i\in \chi_W(v)$ such that $d_{i,j}(v)$ isn't a stump.
	A tree in superposition $T$ is \emph{proper} if it is reduced and if $v\in V(T)$ such that $\abs{\vertexinput(v)}=1$, then this incoming edge is an input edge for the tree.
\end{definition}

We will write $\propertreegroupoid$ for the full subgroupoid spanned by proper trees, and $\coretreegroupoid$ for the full subgroupoid of $\supertreegroupoid$ spanned by proper trees of height $2$ or less. Here, height is the maximal number of vertices along any path.

Given a tree $T$ in superposition with leading vertex $v$, we have the usual decomposition of trees \[T=C_v\circ (T_i)_{i\in \operatorname{In}(v)}.\] Here, $C_v$ is a corolla tree with single vertex $v$, and $\circ$ is tree grafting. For operads $P$ and $Q$, we define the following spaces inductively.
For the trivial tree with no vertices we set $F(P,Q)(\vert):
	=\ast$ and for an arbitrary tree in superposition $T$, we have \[F(P,Q)(T):= P(\chi_W(v))\times Q(\chi_B(v))\times \left(\prod_{i\in \vertexinput(v)}
	F(P,Q)(T_i)\right).
\] 

Moreover, for $\Ord{n}\in\fOrd$ and operads $P,Q$, we will write \begin{align*}
	\vF(P,Q)(n) & := \colim_{T\in\supertreegroupoid(n)}
	F(P,Q)(T)=\coprod_{[T]\in \text{iso}(\supertreegroupoid(n))}F(P,Q)(T) \\
	\propergenerators(P,Q)(n) & := \colim_{T\in\propertreegroupoid(n)}
	F(P,Q)(T)=\coprod_{[T]\in \text{iso}(\propertreegroupoid(n))}F(P,Q)(T) \\
	\coregenerators(P,Q)(n)   & := \colim_{T\in \coretreegroupoid(n)}
	F(P,Q)(T) = \coprod_{[T]\in\operatorname{iso}(\coretreegroupoid(n))}F(P,Q)(T)
\end{align*}

Moreover, we have equivariant maps defined by \begin{align*}
	q_\vert: F(P,Q)(\vert) & \to (P\otimes Q)(1)                \\
	\ast                      & \mapsto \id\shortintertext{and,}
	q_{C_v}: F(P,Q)(C_v)   & \to (P\otimes Q)(\chi_W(v)\times \chi_B(v)) \\
	(f,g)                     & \mapsto f\otimes g,
\end{align*} which we extend inductively by \[q_T=q_{C_v}\circ \left(\prod_{i\in\vertexinput(v)}q_{T_i}\right).\]

This gives us a surjective map of $G$-collections \[q: \vF(P,Q)\to P\otimes Q.\]

A consequence of the interchange relations is that this restricts to a surjective map of $G$-collections \[q:\propergenerators(P,Q)\to P\otimes Q.\]

We will also make use of the following elementary lemma. A proof of which can be found in \cite*[Proposition 4.8]{fiedorowiczAdditivityTheoremInterchange2015}

\begin{lemma}\label{lem:unaryinjection}
	For two reduced operads $P$ and $Q$, there is an isomorphism \[\left(P\otimes Q\right)(1)\cong P(1)\times Q(1).\] A consequence of this is that if we have two injective interchanging maps of reduced operads \begin{align*}
		f: P&\hookrightarrow R \\ g: Q&\hookrightarrow R
	\end{align*} such that $f(1)\cap g(1)= \set{\id}$, then the unary component of the induced map \[(f\otimes g)(1): \left(P\otimes Q\right)(1) \hookrightarrow R(1) \] is injective.
\end{lemma}

\section{Equivariant Framed Little Disks}\label{sec:framedlittledisks}

One difference between the little cube operads $\mathcal{C}_n$ and little disks $\mathcal{D}_n$ -- besides the shapes -- is that the embeddings themselves are slightly different kinds of maps.
The little cubes are rectilinear embeddings while the little disks are dilations.
We will take the view point that rectilinear embeddings are just products of one dimensional dilation embeddings.
This viewpoint makes more sense in the context of $G$-representations as the natural decomposition of a $G$-representation is into irreducible representations which may not be $1$-dimensional.
The kinds of embeddings we will want to look at are products of dilation maps defined on sub $G$-representations. 

Let us set up some notation to better talk about these maps.

\begin{definition}
    Given a vector space $V$, write $\subspace(V)$ for the poset of all subspaces of $V$ under inclusion.
    A (finite orthogonal) \emph{decomposition} is a function $V_\bullet: J\to \subspace(V)$ such that $\bigoplus_{j\in J}V_j = V $ and $J$ is a finite set.
    When we need to make use of the indexing set $J$, we will sometimes say $V_\bullet^J$ is a decomposition of $V$ with the understanding that $J$ is the finite indexing set associated to it.
    The \emph{trivial decomposition} $V_{tr}$ is just one where we treat $tr$ as a singleton set and $V_{tr}=V$.
\end{definition}

We can construct a poset of decompositions of $V$, denoted by $\decomposition(V)$, where for decompositions $W_\bullet^J$ and $V_\bullet^K$ of $V$, we have $ W_\bullet^J \leq V_\bullet^K $ if for all $k\in K$, there exists $j\in J$ such that $W_j \subseteq V_k$.
That is, coarser decompositions are larger.

If $V$ is a $G$-representation, then $\decomposition(V)$ gets a $G$-action by $g\cdot V_i := gV_i$, and we will call any decomposition fixed under this action a $G$-decomposition.
i.e., each $V_i$ is a sub-$G$-representation.

\begin{remark}
    Notice that if $W_\bullet \leq V_\bullet$ and $W_\bullet$ is a $G$-decomposition, then so is $V_\bullet$.
\end{remark}

\begin{definition}
    Given a decomposition $V_\bullet^J\in \decomposition(V)$, we have a topological group \[\orth(V_\bullet) := \prod_{i\in J}\orth(V_i),\] this has the obvious action on $V$.
    When $V$ is a $G$-representation and $V_\bullet$ is a $G$-decomposition, then $\orth(V_\bullet)$ becomes a $G$-topological group where $G$ acts via conjugation. Note that this agrees with the usual orthogonal group notation as $\Orth{V_{tr}}=\Orth{V}$.

    Since we are assuming finite decompositions, we have obvious embeddings \[\Orth(V_\bullet) \hookrightarrow \Orth(V).
    \] In fact, given decompositions $V_\bullet \leq V_\bullet^{\prime}
    $ of $V$ we have embeddings\[\orth(V_\bullet)\hookrightarrow \orth(V_\bullet^{\prime}).
    \]
\end{definition}

\begin{definition}{}{}
    For a vector space $V$, and finite dimensional subspace $W\subseteq V$, we have a topological group given by \[\Lambda_{W}(V) := \set[\bigg]{\vec{w}\mapsto \alpha\vec{w} \given \alpha\in \mathbb{R}_{>0}}\subseteq \GL(W).
    \]
    For each inclusion $W\subseteq W^\prime$ there exists a continuous injective homomorphisms \begin{align*}
        \Lambda_W(V)                     & \hookrightarrow\Lambda_{W^\prime}(V)                       \\
        (\vec{w} \mapsto \alpha \vec{w}) & \mapsto (\vec{w}^\prime \mapsto \alpha \vec{w}^{\prime} ).
    \end{align*}
    We then define the topological group of \emph{dilations} on $V$ by
    \[\Lambda(V):=\colim_{W \subseteq V}\Lambda_W(V).
    \]

    For $V_\bullet\in\decomposition(V)$, we also define \[\Lambda(V_\bullet)= \prod_{i\in I}\Lambda(V_i)\]

    If $V$ is a $G$-representation and $V_\bullet$ a $G$-decomposition, then conjugation on $\Lambda(V_\bullet)$ makes sense and is just a trivial action.
    In contrast to the orthogonal groups, $\Lambda$ is a \emph{contravariant functor} with respect to decompositions.
    That is, given decompositions $V_\bullet \leq V_\bullet^{\prime}$ of $V$ then we have embeddings \[\Lambda(V_\bullet^{\prime})\hookrightarrow \Lambda(V_\bullet).
    \]
\end{definition}
We will also make use of the following submonoid of $\Lambda(V_\bullet)$.
\begin{definition}
    Let $\Lambda(V_\bullet)^0$ denote the submonoid of all elements of $\Lambda(V_\bullet)$ with all eigenvalues in the range $(0,1]$.
\end{definition}

\begin{definition}
    A \emph{pair} of $G$-decompositions $(W_\bullet,V_\bullet)$ are $G$-decompositions such that $W_\bullet \leq V_\bullet$. We define the poset of pairs $\decomposition^2(V)$ where we will set $(W_\bullet,V_\bullet) \leq (W_\bullet^{\prime},V_\bullet^{\prime})$ if we have \[W_\bullet \leq W_\bullet^{\prime}\; \text{and}\; V_\bullet^{\prime} \leq V_\bullet.\]
    Make note of the change of direction between $V_\bullet$ and $V_\bullet^{\prime}$ in this definition.
\end{definition}

\begin{definition}
    Given a pair $(W_\bullet, V_\bullet)$, we can construct the \emph{$G$-dilation group} \[\DL(W_\bullet, V_\bullet) := \orth(W_\bullet)\cdot \Lambda(V_\bullet) \subseteq \GL(V).
    \]
\end{definition}

This construction is functorial: If $(W_\bullet,V_\bullet) \leq (W_\bullet^{\prime},V_\bullet^{\prime})$ then \[\DL(W_\bullet^{\prime}, V_\bullet^{\prime})\hookrightarrow \DL(W_\bullet, V_\bullet).\]

Note we require that $W_\bullet \leq V_\bullet$ for $\DL(W_\bullet,V_\bullet)$ to be a well-defined group as in this case, the elements of $\orth(W_\bullet)$ and $\Lambda(V_\bullet)$ commute past each others.

The following is an equivariant variation of \cite[Definition 4.9]{dwyerConfigurationSpacesProducts2018}

\begin{definition}\label{def:dilationrepresentation}
    Given a topological $G$-group $\mathcal{G}$, a \emph{dilation representation structured by the pair of $G$-decompositions $(W_\bullet,V_\bullet)$} is a continuous $G$-homomorphism \[\rho: \mathcal{G}\to \DL(W_\bullet,V_\bullet)\] such that \[\im(\rho) = (\im(\rho)\cap \orth(W_\bullet))\cdot \Lambda(V_\bullet).
    \]

    If $V_\bullet = V_{tr}$, then we will say $\rho$ is \emph{spherical}.
    
\end{definition}

\begin{remark}
    \Cref{def:dilationrepresentation} doesn't reduce to the Definition 4.9 in \cite{dwyerConfigurationSpacesProducts2018}. Dwyer-Hess-Knudsen consider rotated rectilinear embeddings which to capture here would require us to consider any pair of decompositions $(W_\bullet,V_\bullet)$ and not just those where $W_\bullet\leq V_\bullet$. This small change would vastly increase the complexity of our arguments (and also force us to introduce many more conditions) and so we don't consider maps of this type in this paper. 
\end{remark}

\begin{convention}\label{ass:topologicalgroups}
    \Cref{def:dilationrepresentation} is a bit too general as stated. We will need to assume a number of extra properties about our topological groups $\vG[cal]$:
    \begin{enumerate}
        \item There exists a $G$-homomorphism \[s_\rho:\Lambda(V_\bullet)\to \mathcal{G}\] that is a partial section to $\rho$;
        \item the $G$-group $\mathcal{G}$ is paracompact Hausdorff.
    \end{enumerate}
\end{convention}

We will also need to assume the following about dilation representations.

\begin{convention}\label{filtered group conventions}
    Let $V$ be a $G$-representation and $\rho: \vG[cal]\to \DL(V_\bullet,V_\bullet^\prime)$ a dilation representation on $V$. We will assume that there exists an increasing filtration of finite dimensional $G$-representations $\set{0} \subseteq W_1 \subseteq W_2 \subseteq \dots$ with $\cup_i W_i = V$, such that for each of the $G$-groups $\vG[cal,i]^\rho$ defined by the pullbacks \[\begin{tikzcd}
        \vG[cal,i]^\rho \ar[r]\ar[d]\ar[rd,phantom," " pullback,very near start]& \DA(V_\bullet\cap W_i, V_\bullet^\prime\cap W_i) \ar[d]\\
        \vG[cal]^\rho \ar[r]& \DA(V_\bullet, V_\bullet^\prime)
    \end{tikzcd}\]
    the multiplication maps are proper maps.
\end{convention}

We can identify $\GA(V)$ as the semidirect product $\GL(V)\ltimes V$ where multiplication is given \[(f,v)\cdot (g,w) = (fg,v+f(w)),\] and the action of $\GA(V)$ on $V$ is given by \[(f,v)\cdot w = v+f(w).
\] As a subgroup, we have the affine dilation transformations $\DA(W_\bullet,V_\bullet)$ which is $\DL(W_\bullet,V_\bullet)\ltimes V$. 

We can extend this construction for general dilation representations.
For a dilation representation $\rho:\mathcal{G}\to \DL(W_\bullet,V_\bullet)$, we can take the $G$-semi-direct product $\mathcal{G}^\rho:=\mathcal{G}\ltimes_{\rho} V$.
We then have an induced \emph{affine dilation representation} \begin{align*}
    \hat{\rho}: \mathcal{G}^\rho & \to \DA(W_\bullet,V_\bullet) \\ (f,v) &\mapsto (\rho(f), v)
\end{align*}
which is a $G$-homomorphism.
The unit of $\mathcal{G}^\rho$ will be denoted by $\id_{\rho}:=(\id_{\mathcal{G}},0)$.
We will also often omit $\hat{\rho}$ and just treat elements of $\mathcal{G}^\rho$ as affine transformations of $V$.
That is, for $x\in \mathcal{G}^\rho$ we write $x(v)$ for $\hat{\rho}(x)(v)$.
We will also generally identify $\mathcal{G}$ with the copy of itself $\mathcal{G}\times \set{0} \subseteq \mathcal{G}^\rho$, so we can view the partial section $s_\rho$ as a map into $\mathcal{G}^\rho$.

\begin{definition}{}{}
    Let $V$ be a $G$-representation.
    A \emph{star domain} $S$ is a bounded open subset $S\subseteq V$ such that \begin{enumerate}
        \item $0\in S$,
        \item for all points $x\in S$ and $t\in[0,1]$, we have that $tx\in S$.
    \end{enumerate}
\end{definition}

\begin{remark}
    We don't require that star domains are $G$-invariant.
\end{remark}

\begin{definition}{}{}
    Given a dilation representation $\rho: \mathcal{G}\to \DL(W_\bullet,V_\bullet)$, and star domain $S \subseteq V$, we have the following $G$-monoids:
    \[\mathcal{G}^\rho(S):=\set{x\in \mathcal{G}^\rho\given (g\cdot x)(S) \subseteq S \text{ for all }g\in G}\]
\end{definition}

\begin{definition}{}{}
    Given a dilation representation $\rho:\mathcal{G}\to \DL(W_\bullet,V_\bullet)$ and star domain $S$ of $V$, we define the \emph{$\rho$-framed ambient little star operad of shape $S$} by \[\ambientoperad^\rho(S):=\mathscr{O}(\mathcal{G}^\rho(S)).
    \]
    We will also use $\ambientoperad^\rho := \mathscr{O}(\mathcal{G}^\rho)$, which we call the \emph{$\rho$-framed universal ambient operad}.
\end{definition}

\begin{definition}{}{}
    Given a dilation representation $\rho: \mathcal{G}\to \DL(W_\bullet,V_\bullet)$, and star domain $S$, we define the \emph{$\rho$-framed little star operad} \(\staroperad^{\rho}(S)\) as the suboperad of $\ambientoperad^{\rho}(S)$ where $x\in \ambientoperad^{\rho}(S)$ is in \(\staroperad^{\rho}(S)\) if for all $i\neq j \in \ar(x)$, and $g\in G$ we have that \[(g\cdot x_i)(S)\cap(g\cdot x_j)(S)=\emptyset.
    \]
\end{definition}

\begin{example}
    Write $B_V$ for the open unit disk centered at the origin of a $G$-representation $V$. Given the trivial group $\mathbb{1}\subseteq GL(V)$, the little star operad $\staroperad^{\mathbb{1}}(B_V)$ is then the operad of little $V$-disks $\mathcal{D}_V$.
\end{example}

\begin{example}
    The operad $\staroperad^{\operatorname{SO}(n)}(B_{\mathbb{R}^n})$ is the framed little disk operads.
\end{example}

\section{Comparisons Between Framed Disk Operads}\label{sec:comparison}
Any dilation representation $\rho: \mathcal{G}\to \DL(V_\bullet,V_\bullet^{\prime})$ uniquely splits into the product of $G$-group homomorphisms $\rho(x) =\sigma^\rho(x)\delta^\rho(x) $ where \begin{align*}
    \sigma: & \mathcal{G}\to \orth(V_\bullet) \\ \delta: &\mathcal{G}\to \Lambda(V_\bullet^{\prime}).
\end{align*}
If $J$ and $J^{\prime}$ are the indexing sets of the recompositions $V_\bullet$ and $V_\bullet^{\prime}$ respectively, then we will also use the notation \begin{align*}
    \sigma_j := \restr{\sigma}{V_j} \text{ and } & \delta_j := \restr{\delta}{V^{\prime}_j}.
\end{align*}

Given a dilation representation $\rho : \mathcal{G}\to \DL(W_\bullet, V_\bullet)$ and $V_\bullet\leq V_\bullet^{\prime}$, then we have an inclusion $i:\DL(W_\bullet, V_\bullet^{\prime})\hookrightarrow \DL(W_\bullet, V_\bullet)$.
We can pull back the dilation representation to get a closed $G$-subgroup $i^\ast \mathcal{G} \subseteq \mathcal{G}$ and a dilation representation \[i^\ast\rho:i^\ast \mathcal{G} \to \DL(W_\bullet, V_\bullet^{\prime}).
\] When $V_\bullet^{\prime}=V_{tr}$, we will write $\mathcal{G}_{tr}$ and $\rho_{tr}$ for $i^\ast \mathcal{G}$ and $i^\ast\rho$ respectively. 
The $G$-group $\mathcal{G}_{tr}$ is easily described using the function $\delta$ above: $$\mathcal{G}_{tr}= \set[\bigg]{x\in \mathcal{G}\given \delta(x)=\lambda \id_{\mathcal{G}}\text{ for some }\lambda\in (0,\infty)}.
$$

The inclusion $\vG[cal]_{tr}\to \vG[cal]$ induces a closed embedding of operads \[\staroperad^{\rho_{tr}}(S)\hookrightarrow \staroperad^{\rho}(S)\] where $S$ is some star domain of $V$. We will first show that this is a weak equivalence of $G$-operads.

\begin{lemma}
    Let $\rho: \mathcal{G}\to \DL(V_\bullet,V_\bullet^{\prime})$ be a dilation representation.
    There exists a $G$-invariant function $\lambda: \mathcal{G}\to \Lambda(V_\bullet^{\prime})^0$ such that for any $x\in \mathcal{G}$ we have that \[x\cdot (s_\rho\lambda)(x)\in \mathcal{G}_{tr}.
    \]
\end{lemma}

\begin{proof}
    Let $J$ be the indexing set for the decomposition of $V_\bullet^{\prime}$.
    For any $x\in \mathcal{G}$ and $j\in J$, the linear transformations $\delta_j(x)$ is just scalar multiplication.
    We will abuse notation and consider $\delta_j(x)\in (0,\infty)$.
    Note that these functions are $G$-invariant as $G$ ultimately acts by conjugation and $V_\bullet^{\prime}$ is a $G$-decomposition.
    The required function is then simply constructed as \[\lambda(x):=\bigoplus_i \frac{\min_{j}\delta_j(x)}{\delta_i(x)}\id_{V_i}\]
\end{proof}

We extend the function $\lambda$ of this lemma to a homotopy by setting \begin{align*}
    \lambda: \mathcal{G}\times [0,1]\to \Lambda(V_\bullet^{\prime})^0 \\ 
    \lambda(x,t) := (1-t)\id_V + t\lambda(x)
\end{align*}

The following is then straightforward to check.

\begin{lemma}\label{spherical equivalence lemma}
    For any star domain $S$, the homotopy \begin{align*}
        H:\staroperad^\rho(S)\times [0,1]&\to \staroperad^\rho(S) \\ H((x_i),t) &:= (x_i\circ s_\rho \lambda(x_i,t))_i
    \end{align*} 
    exhibits $\staroperad^{\rho_{tr}}(S)$ as a strong deformation retract (of $G$-collections) of $\staroperad^\rho(S)$. In particular, $\staroperad^{\rho_{tr}}(S)$ is weakly equivalent to $\staroperad^{\rho}(S)$ as $G$-operads.
\end{lemma}

We will now switch focus to showing that the homotopy type of $\staroperad^\rho(S)$ is independent of the particular star shape $S$. As explained in the introduction, we will only consider star shapes that are products of disks. 

Given a decomposition $W_\bullet:J\to \subspace(V)$ and function $\lambda:J\to \mathbb{R}_{>0}$, we will write $B_{W_\bullet}(v;\lambda_\bullet):= \bigtimes_{j\in J}B_{V_j}(\pr_{V_j}(v);\lambda_j)$. Note that if we have decompositions \[V_\bullet \leq W_\bullet \leq V_\bullet^{\prime}\] then any $f\in\DA(V_\bullet,V_\bullet^{\prime})$ is such that $$f\left(B_{W_\bullet}(v,\lambda_\bullet)\right)= B_{W_\bullet}(f(v),\mu_\bullet\lambda_\bullet)$$ where $\mu_j$ are some positive scalars.

For the next two lemmas, we will fix a dilation representation $\rho:\mathcal{G}\to \DL(V_\bullet,V_\bullet^{\prime})$ and decompositions \[V_\bullet \leq W_\bullet,W_\bullet^{\prime}\leq V_\bullet^{\prime}.\] We will also set $B:=B_{W_\bullet}(0,\epsilon_\bullet) $, $B^{\prime}=B_{W_\bullet^{\prime}}(0,\epsilon_\bullet^{\prime})$ and assume $B \subseteq B^{\prime}$

\begin{definition}{}{}
    Given star domains $S,T$ of $V$ we will write $\staroperad^{\rho}(S,T)$ for the intersection $\staroperad^{\rho}(S)\cap \staroperad^{\rho}(T)$.
\end{definition}
\sloppy{
\begin{lemma}\label{lem:comparisonfromlarger}
    There is a deformation retract of $\staroperad^\rho(B^{\prime})$ onto $\staroperad^\rho(B,B^{\prime})$ as $G$-collections.
    In particular, this implies that the inclusion of operads \[\staroperad^\rho(B,B^{\prime})\hookrightarrow \staroperad^\rho(B^{\prime})\] is a weak equivalence of $G$-operads.
\end{lemma}
}
\begin{proof}
    We will use \cref{lem:deformationcorrection}.
    Define a function \begin{align*}
        H: \staroperad^\rho(B^{\prime})\times[0,1)\to \staroperad^\rho(B^{\prime}) \\ H(x,t) := s_\rho\left(1-t\right)\circ x
    \end{align*}

    Before we verify that the conditions of \cref{lem:deformationcorrection} hold, let us make the following important observation.
    Because we have $B \subseteq B^{\prime}$, if $x\in \staroperad(B^{\prime})$ then $$(g\cdot x_i)(B)\cap (g\cdot x_j)(B)\subseteq (g\cdot x_i)(B^{\prime})\cap (g\cdot x_j)(B^{\prime}) =\emptyset$$ for all $g\in G$, and $i\neq j\in \ar(x)$.
    Hence, we have that for $x\in \staroperad(B^{\prime})$ that $x\in\staroperad(B,B^{\prime})$ if and only if $(g\cdot x_i)(B) \subseteq B$ for all $g\in G$, and $i\in \ar(x)$.

    Now, the suboperad $\staroperad^\rho(B,B^{\prime})$ is closed in $\staroperad^\rho(B^\prime)$ from \Cref{suboperads are closed lemma}.
    Let us now verify the needed conditions. In what follows we will omit the section map $s_\rho$ from notation and interpret scalar multiplication as acting through this map.

    \emph{Condition (1): There exists $t\in [0,1)$ such that $H(x,t)\in \staroperad^\rho(B,B^{\prime})$. }
    Since $B$ and $B^{\prime}$ are centered at $0$, we have that there exists $\lambda\in(0,1]$ such that $\lambda B^{\prime} \subseteq B$.
    Let $t=1-\lambda$.
    Then we have for all $g\in G$, $i\in \ar(x)$ that \[\left(g\cdot H(x,t)_i\right)(B^{\prime})= \lambda (g\cdot x_i)(B^{\prime}) \subseteq \lambda B^{\prime} \subseteq B.
    \] Hence, $$(g\cdot H(x,t)_i)(B) \subseteq (g\cdot H(x,t)_i)(B^{\prime}) \subseteq B$$ for all $i\in \ar(x)$, and $g\in G$. i.e., $H(x,t)\in \staroperad^\rho(B,B^{\prime})$ and condition (1) holds.

    \emph{Condition (2): If $H(x,t)\in \staroperad^\rho(B,B^{\prime})$, then $H(x,s)\in \staroperad^\rho(B,B^{\prime})$ for all $s\geq t$.}
    Observe the following for all $g\in G$, $i\in \ar(x)$: \begin{align*}
        (g\cdot H(x,s)_i)(B) & =(1-s)(g\cdot x)_i(B) \\ &= \frac{1-s}{1-t}(1-t)(g\cdot x)_i(B) \\ &= \frac{1-s}{1-t}(g\cdot H(x,t)_i)(B).
    \end{align*}
    Hence, if $H(x,t)\in \staroperad^\rho(B,B^{\prime})$, then for all $s\geq t$, we also have that \[(g\cdot H(x,s)_i)(B)=\frac{1-s}{1-t}(g\cdot H(x,t)_i)(B) \subseteq \frac{1-s}{1-t}B \subseteq B.
    \] Therefore, $H(x,s)\in \staroperad^\rho(B,B^{\prime})$ and condition (2) holds.

    \emph{Condition (3): The set $\left(\set{x}\times [0,1)\right) \cap H^{-1}(\partial \staroperad^\rho(B,B^{\prime}))$ is a singleton.}
    First observe that if $x\in \staroperad^\rho(B,B^{\prime})$ and there exists a $0<\mu<1$ such that $(g\cdot x)_i(B) \subseteq \mu B$ for all $g\in G$ and $i\in \ar(x)$ then $x\in \operatorname{Int}(\staroperad^\rho(B,B^{\prime}))$.

    Now, suppose that $\left(\set{x}\times [0,1)\right) \cap H^{-1}(\partial \staroperad^\rho(B,B^{\prime}))$ isn't a singleton.
    In particular, there exists $0<\lambda_1<\lambda_2$ such that $$H(x,\lambda_1),H(x,\lambda_2)\in \partial \staroperad^\rho(B,B^{\prime}).
    $$
    Since we have that \begin{align*}
        H(x,\lambda_2)= \frac{1-\lambda_2}{1-\lambda_1}H(x,\lambda_1)
    \end{align*}
    It follows that for all $i\in \ar(x)$ that \[(H(x,\lambda_2))_i(B) = \frac{1-\lambda_2}{1-\lambda_1}(H(x,\lambda_1))_i(B) \subseteq \frac{1-\lambda_2}{1-\lambda_1}B \subseteq  B.
    \] Since $0<\frac{1-\lambda_2}{1-\lambda_1}<1$, we deduce that $H(x,\lambda_2)$ is in the interior of $\staroperad(B,B^{\prime})$ which is a contradiction.
    Hence, condition (3) holds.

    In summary, we have verified the conditions for \cref{lem:deformationcorrection}, and so we conclude that there is a continuous $G$-invariant map \[\phi:\staroperad^\rho(B^{\prime})\to [0,1)\] such that the map \[\tilde{H}(x,t):=H(x,\phi(x)t)\] is a deformation retract of $\staroperad^\rho(B^{\prime})$ onto $\staroperad^\rho(B,B^{\prime})$ as $G$-collections.
\end{proof}

\begin{lemma}\label{lem:comparisonfromsmaller}
    There is a deformation retract of $\staroperad^\rho(B)$ onto $\staroperad^\rho(B,B^{\prime})$ as $G$-collections.
    In particular, this implies that the inclusion of operads \[\staroperad^\rho(B,B^{\prime})\hookrightarrow \staroperad^\rho(B)\] is a weak equivalence of $G$-operads.
\end{lemma}
\begin{proof}
    The proof of this lemma mirrors that of \cref{lem:comparisonfromlarger} except the homotopy under consideration comes from right multiplication of scalers, and we have to deal with an intersection condition instead of an inclusion one.
    
    This time, let us define \begin{align*}
        H: \staroperad^\rho(B)\times[0,1)\to \staroperad^\rho(B) \\ H(x,t) := x\circ \left(s_\rho(1-t)\right)_{i\in\ar(x)}.
    \end{align*}
    Because we are only considering open balls in $B$ and $\vB[']$, we know that if $f\in\DA(V_\bullet,V_\bullet^\prime)$ and $f(B) \subseteq B$, then $f(\vB[']) \subseteq \vB[']$. A consequence of this is that for $x\in \staroperad^\rho(B)$, we automatically get that $x\in\ambientoperad^\rho(B^{\prime})$, and so we have $x\in\staroperad^\rho(B,B^{\prime})$ if and only if for all $g\in G$ and $i\neq j\in\ar(x)$ that \[(g\cdot x_i)(B^{\prime})\cap (g\cdot x_j)(B^{\prime})=\emptyset.\]

    \emph{Condition (1): There exists $t\in [0,1)$ such that $H(x,t)\in \staroperad^\rho(B,B^{\prime})$. }
    Again, there exists $\lambda\in(0,1]$ such that $\lambda B^{\prime} \subseteq B$.
    Let $t=1-\lambda$.
    Then we have for all $g\in G$, $i\neq j\in \ar(x)$ that 
    \begin{align*}
        \left(g\cdot H(x,t)_i\right)(B^{\prime})\cap \left(g\cdot H(x,t)_j\right)(B^{\prime}) &= \left(g\cdot x_i\right)(\lambda B^{\prime})\cap \left(g\cdot x_j\right)(\lambda B^{\prime}) \\ &\subseteq \left(g\cdot x_i\right)(B)\cap \left(g\cdot x_j\right)(B) \\ &= \emptyset. 
    \end{align*}
     Hence, $H(x,t)\in \staroperad^\rho(B,B^{\prime})$ and condition (1) holds.

    \emph{Condition (2): If $H(x,t)\in \staroperad^\rho(B,B^{\prime})$, then $H(x,s)\in \staroperad^\rho(B,B^{\prime})$ for all $s\geq t$.}
    Using the same algebraic trick (except on the right side) as in the last lemma, we have that \[(g\cdot H(x,s)_i)(B^{\prime}) = (g\cdot H(x,t)_i)\left(\frac{1-s}{1-t}B^{\prime}\right). \]
    Therefore, if $H(x,t)\in \staroperad^\rho(B,B^{\prime})$ and $s\geq t$, then we have for $i\neq j\in\ar(x)$\begin{align*}
        &(g\cdot H(x,s)_i)(B^{\prime})\cap (g\cdot H(x,s)_j)(B^{\prime}) \\ &= (g\cdot H(x,t)_i)\left(\frac{1-s}{1-t}B^{\prime}\right) \cap (g\cdot H(x,t)_j)\left(\frac{1-s}{1-t}B^{\prime}\right)\\ &\subseteq  (g\cdot H(x,t)_i)\left(B^{\prime}\right)\cap (g\cdot H(x,t)_j)\left(B^{\prime}\right) \\ &=\emptyset.
    \end{align*}
    So condition (2) holds.

    \emph{Condition (3): The set $\left(\set{x}\times [0,1)\right) \cap H^{-1}(\partial \staroperad^\rho(B,B^{\prime}))$ is a singleton.} In comparison to \cref*{lem:comparisonfromlarger}, for $x\in \staroperad^\rho(B,B^{\prime})$, if there exists $\mu>1$ such that $$(g\cdot x)_i(\mu B^{\prime})\cap (g\cdot x)_j(\mu B^{\prime})  =\emptyset$$ for all $g\in G$ and $i\neq j\in \ar(x)$ then $x\in \operatorname{Int}(\staroperad^\rho(B,B^{\prime}))$.

    Now, suppose that $\left(\set{x}\times [0,1)\right) \cap H^{-1}(\partial \staroperad^\rho(B,B^{\prime}))$ isn't a singleton.
    In particular, there exists $0<\lambda_1<\lambda_2$ such that $$H(x,\lambda_1),H(x,\lambda_2)\in \partial \staroperad^\rho(B,B^{\prime}).
    $$
    Since we have that \begin{align*}
        H(x,\lambda_1)= H(x,\lambda_2)\circ \frac{1-\lambda_1}{1-\lambda_2}
    \end{align*}
    It follows that for all $i\neq j\in \ar(x)$, $g\in G$ that 
    \begin{align*}
        &(g\cdot H(x,\lambda_2)_i)\left(\frac{1-\lambda_1}{1-\lambda_2}B^{\prime}\right)\cap (g\cdot H(x,\lambda_2)_j)\left(\frac{1-\lambda_1}{1-\lambda_2}B^{\prime}\right) \\ &= (g\cdot H(x,\lambda_1)_i)\left(B^{\prime}\right)\cap (g\cdot H(x,\lambda_1)_j)\left(B^{\prime}\right) \\ &= \emptyset.
    \end{align*}
    Since $\frac{1-\lambda_1}{1-\lambda_2}>1$, we deduce that $H(x,\lambda_2)$ is in the interior of $\staroperad(S,T)$ which is a contradiction.
    Hence, condition (3) holds.

\end{proof}

The following is the main theorem of this section.

\begin{theorem}\label{thm:maincomparison}
    Let $\rho:\mathcal{G}\to \DL(V_\bullet,V_\bullet^{\prime})$ be a dilation representation, and \[V_\bullet \leq W_\bullet,W_\bullet^{\prime}\leq V_\bullet^{\prime}\] $G$-decompositions of $V$. For \emph{any} $B:=B_{W_\bullet}(0,\epsilon_\bullet) $, and $B^{\prime}=B_{W_\bullet^{\prime}}(0,\epsilon_\bullet^{\prime})$,
    there is a zigzag of weak equivalences between $\staroperad^\rho(B)$ and $\staroperad^\rho(B^{\prime})$.
\end{theorem}
\begin{proof}
    There exists a decomposition $W_\bullet^{\prime\prime}$ such that \[V_\bullet \leq W_\bullet^{\prime\prime} \leq V_\bullet^{\prime},\] and $B^{\prime\prime}:= B_{W_\bullet^{\prime\prime}}(0,\epsilon_\bullet^{\prime\prime})$ where $B^{\prime\prime} \subseteq B\cap B^{\prime}$.
    
    Using the previous two lemmas, the following is then a zigzag of weak equivalence:
    \[\staroperad^\rho(B) \leftarrow \staroperad^\rho(B,B\cap B^{\prime\prime}) \rightarrow \staroperad^\rho(B^{\prime\prime}) \leftarrow \staroperad^\rho(B^{\prime\prime},B^{\prime}) \rightarrow \staroperad^\rho(B^\prime).
    \]
\end{proof}

\begin{remark}
    \Cref{thm:maincomparison} can be viewed as a generalization of the well known fact that the little cube operads $\mathcal{C}_n$ and little disk operads $\mathcal{D}_n$ are weakly equivalent. This can be seen by taking just the trivial dilation representation and $G=\set{e}$ the trivial group. Then the unit disk and unit cube are both of the form $B(0,\epsilon_\bullet)$ for different decompositions. 
\end{remark}

\section{Divisibility in Operads}\label{sec:divisibility}
The difficult part of proving the additivity theorem is showing certain maps from the Boardman-Vogt tensor are injective.
This is intimately linked to understanding how divisors behave in the component operads.
By a divisor, we mean the following.

\begin{definition}{}{}
    Let $P$ be an operad and $x,y,q^i\in P$ such that \[y=x\underset{\alpha}{\circ}(q^i)_{i\in\ar{x}}.\]
    We will say that $x$ is a \emph{(left) divisor of $y$} and $q^i$ are the \emph{quotients}.
    We will also say that $\alpha$ \emph{structures the division}.
\end{definition}

Quotients are not necessarily unique in general, however, they are for the operads that we consider and this turns out to be a key property.

\begin{definition}{}{}
    An operad $P$ is \emph{left-cancellable} if for any $y\in P$ and two divisions \[y= x\circ (q_1^i)_{i\in\ar(y_1)} = x\circ (q_2^i)_{i\in\ar(y_1)}\] which are structured by the same map $\alpha: \ar(y)\to\ar(x)$, then we must have that $q_1^i = q_2^i$ for all $i\in\ar(x)$.
    Note that a left-cancellable operad is necessarily a reduced operad.
\end{definition}

The reason that an operad being left-cancellable is a key condition is that it allows inductive arguments to be made on the cardinality of an elements indexing set. This forms the main idea of the proof technique of \Barata-\Moerdijk~\cite{barataAdditivityLittleCubes2022}. For instance, if we want to show that an operad map $f:P\to Q$ was injective, then suppose we have $x,y\in P_n$ such that $f(x)=f(y)$ and $f$ is injective on $P_{<n}$. Furthermore, assume that $x,y$ share a common divisor $z$ where $\abs{\ar(z)}>1$. i.e., \begin{align*}
    x &= z\circ (q^i)_{i\in\ar(z)} \\
    y &= z\circ (r^i)_{i\in\ar(z)}
\end{align*}
for some $q^i,r^i\in P_{<n}$. Then we have that \begin{align*}
    f(x) &= f(y) \\ 
    f(z) \circ (f(q^i))_{i\in\ar(z)} &= f(z) \circ (f(r^i))_{i\in\ar(z)} 
\end{align*}
which, assuming $Q$ is left-cancellable, we get for all $i\in\ar(z)$ that \begin{align*}
    f(q^i) &= f(r^i) .
\end{align*}
So then $q^i=r^i$ since we assumed injectivity on $P_{<n}$, and we then deduce that $x=y$. 

The proof that the map $\mathcal{C}_n\otimes \mathcal{C}_m\to \mathcal{C}_{n+m}$ is injective essentially amounts to the above argument and showing that common divisors exist. Unfortunately, such common divisors don't exist for little disk operads. Instead, our argument is a bit more involved, and we will instead show that there exists a chain of divisors.

Since $\ambientoperad^\rho(S)$ lives inside $\ambientoperad^\rho$ which is generated by a \emph{group} and not just a monoid, quotients (when they exist), are unique. It is straightforward to prove that
\begin{lemma}{}{}
    The operads $\ambientoperad^\rho(S)$ and $\staroperad^\rho(S)$ are left-cancellable.
\end{lemma}

Interestingly, whether two elements divide each other in the various star operads can be determined in an entirely geometric way as the next lemma shows.

\begin{lemma}{}{}
    Let $H\leq G$, and suppose $S$ is an $H$-invariant star domain of $G$-representation $V$.
    Write $g_1,\dots, g_n$ for a complete set of representatives of the right coset space $H/G$.
    Then, $x\in\ambientoperad^\rho(S)$ is a divisor for $y\in \ambientoperad^\rho(S)$ in $\ambientoperad^\rho(S)$ if and only if for each $i\in \ar(y)$, there exists $j\in\ar(x)$ such that $$(g_k\cdot y_i)(S) \subseteq (g_k\cdot x_j)(S)$$ for all $g_k$.
\end{lemma}
\begin{proof}
    If $x$ is a divisor of $y$, let $\alpha:\ar(y)\to \ar(x)$ be the map that structures the division with quotients $q^j\in \ambientoperad^\rho(S)$ for $j\in \ar(x)$ such that \[y=x\underset{\alpha}{\circ} (q^j)_{j\in\ar(x)}.\]
    So for each $i\in\ar(y)$, we have that $y_i=x_{\alpha(i)}\circ q^{\alpha(i)}_i$ and for each $g\in G$ we have \[(g\cdot y_i)(S)=(g\cdot x_{\alpha(i)})\circ (g\cdot q^{\alpha(i)}_i)(S) \subseteq (g\cdot x_{\alpha(i)})(S).\]
    Since any $g\in G$ can be written as $h_gg_k$ for some $h_g\in H$ and right coset representative $g_k$, we get that \[h_g(g_k\cdot y_i)(S) \subseteq h_g(g_k\cdot x_{\valpha{i}})(S)\] using $S$ is $H$-invariant.
    Applying the inverse $h_g^{-1}$ gives us the forward direction.

    Conversely, if for each $i\in\ar(x)$, there exists $j\in\ar(y)$ such that $(g_k\cdot y_i)(S) \subseteq (g_k\cdot x_j)(S)$ for all representatives $g_k$.
    We can then build a function $\alpha:\ar(y)\to\ar(x)$ so that $$(g_k\cdot y_i)(S) \subseteq (g_k\cdot x_{\alpha(i)})(S).$$

    Note that as $S$ is $H$-invariant and every element of $g\in G$ can be written in the form $h_gg_k$ for some $h_g\in H$ and representative $g_k$, we have that $$(g\cdot y_i)(S) \subseteq (g\cdot x_{\alpha(i)})(S)$$ for all $g\in G$.

    We now build the quotients to show that $x$ is a divisor of $y$.
    For each $j\in\ar(x)$ define the following element in $\ambientoperad^\rho$: \[q^j:=\begin{cases}
            \ast & \text{if } j\notin\im(\alpha) \\ (x_{j}^{-1}\circ y_i)_{i\in \alpha^{-1}(j)} & \text{if } j\in\im(\alpha).
        \end{cases}\] This is defined so that $y=x\underset{\valpha}{\circ} (q^j)_{j\in\ar(x)}
    $.
    Hence, all we need to show now is that $q^{\alpha(i)}_i\in \mathcal{G}^\rho(S)$ for all $i\in \ar(y)$.
    By applying $(g\cdot x_{\alpha(i)}^{-1})$ to the above inclusion, we get that \[(g\cdot q_i^{\alpha(i)})(S)= (g\cdot x_{\alpha(i)}^{-1})\circ (g\cdot y_i)(S) \subseteq S.\]
    Hence, we are done.
\end{proof}

\begin{lemma}{}{}
    Suppose $x,y\in \ambientoperad^\rho(S)$, $S$ a star domain of $V$, and $x$ divides $y$ in $\ambientoperad^\rho(S)$ with quotients $(q^j)_{j\in \ar(y)}$.
    If $y\in \staroperad^\rho(S)$ then $q^j\in \staroperad^\rho(S)$ for all $j\in \ar(y)$.
\end{lemma}
\begin{proof}
    Let $\alpha: \ar(y)\to \ar(x)$ be the map that structures the division, so for each $i\in\ar(y)$ we have that \[y_i=x_{\alpha(i)}\circ q^{\alpha(i)}_i\] where $q^{\alpha(i)}\in \ambientoperad^\rho(S)$.
    So as elements of $\mathcal{G}^\rho$ we have that \[q^{\alpha(i)}_i=x_{\alpha(i)}^{-1} y_i.\]

    To show that $q^j\in \staroperad^\rho(S)$ all $j$, first note that if $j\notin \im(\alpha)$, then $q^j=\ast$ since the operad is reduced.
    Otherwise, let $j\in \im(\alpha)$, and suppose $i,i^{\prime}\in \alpha^{-1}(j)$ and $i\neq i^{\prime}$.
    Since $y\in\staroperad^\rho(S)$ we know that for all $g\in G$ that \[(g\cdot y_i)(S)\cap (g\cdot y_{i^{\prime}})(S)=\emptyset.\]
    By applying the homeomorphism $(g\cdot x_{\alpha(i)}^{-1})$ we get that \[(g\cdot q_i^j)(S)\cap(g\cdot q_{i^{\prime}}^j)(S)=\emptyset\] and conclude $q^j\in \staroperad^\rho(S)$ all $j$.
\end{proof}

The following is then a critical characterization identifying when elements divide each other. We will often use this result without comment in the future. 

\begin{corollary}{}{}
    If $S$ is a $G$-invariant star domain and $x,y\in \staroperad^\rho(S)$, then $x$ divides $y$ in $\staroperad^\rho(S)$ if and only if for each $i\in \ar(y)$, there exists a $j\in \ar(x)$ such that \[y_i(S) \subseteq x_j(S).\]
\end{corollary}

A consequence of these results is that we can geometrically describe common divisors. Suppose that $a,b \in \staroperad^\rho(S)$ are divisors for some element $x\in \staroperad^\rho(S)$. A common divisor $c\in\staroperad^\rho(S)$ is then any element such that the images $c_i(S)$ fit between $x_j(S)$ and $a_{j^\prime}(S)\cap b_{j^{\prime\prime}}(S)$.
The little cube operads are a particularly nice operad to work with in this respect because the intersections uniquely define a common divisor element. This, however, fails for more general shapes. Since intersections of shapes are so important, we will define some notation now to more easily talk about them in the future.

\begin{definition}\label{def:intersectioncorrespondence}
    Given elements $x,y\in \ambientoperad^\rho(S)$, the \emph{intersection correspondence $\cap_{x,y}$} on $\ar(x)\times\ar(y)$ is given by $$i\cap_{x,y}j\text{ if }x_i(S)\cap y_j(S) \neq \emptyset.$$
    For any set $I \subseteq \ar(x)$, we will define the set \[c_{x,y}(I)=\set[\bigg]{j\in \ar(y)\given i\cap_{x,y} j \text{ for some } i\in I}\] and use the shorthand \[c_{x,y}(i):= c_{x,y}(\set{i}).\]
\end{definition}

\begin{definition}\label{def:intersectionrelation}
    For an element $x\in \ambientoperad^\rho(S)$, the \emph{intersection relation $\cap_{x}$} on $\ar(x)$ is given by $$i\cap_{x}j:= i\cap_{x,x} j.$$
\end{definition}

\section{Additivity of Equivariant Framed Little Disks}\label{sec:additivity}
Let us start this section with an observation.
Given a dilation representation $\rho:\mathcal{G}\to\DL(V_\bullet,V_\bullet^{\prime})$ where $V_\bullet^{\prime}:J\to \subspace{V_\bullet^{\prime}}$.
Recall that we have that \[\DL(V_\bullet\cap V_j^{\prime},V_j^{\prime})\hookrightarrow \DL(V_\bullet,V_\bullet^{\prime}).
\] We will denote the pullback of $\rho$ along this inclusion by \[\rho_j: \mathcal{G}_j\to \DL(V_\bullet\cap V_j^{\prime},V_j^{\prime}).\] It is not difficult to show that we have that $\rho = \times_{j\in J}\rho_j$.
In other words, any dilation breaks into spherical dilations.
We will prove that we have an additivity theorem for spherical dilations and then discuss how the general case works.

Throughout this section we will work with two spherical dilation representations \[\rho:\mathcal{G}\to\DL(V_\bullet,V_{tr})\text{, and }\psi:\mathcal{H}\to\DL(W_\bullet,W_{tr})
\] and open balls $B_V:=B_V(0;\epsilon)$, $B_W:=B_W(0;\epsilon^{\prime})$ for some $\epsilon,\epsilon^{\prime}>0$. There are injective operad maps \begin{align*}
    i_V: \staroperad^\rho(B_V) & \to \staroperad^{\rho\times \psi}(B_V\times B_W) \\ (x_i)                  & \mapsto (x_i\times \id_{\psi})\shortintertext{and,}
    i_W: \staroperad^\psi(B_W) & \to \staroperad^{\rho\times \psi}(B_V\times B_W) \\ (y_j) &\mapsto ( \id_{\rho}\times y_j)
\end{align*} which interchange. This means we get an induced map \[\staroperad^\rho(B_V)\otimes \staroperad^\psi(B_W)\to \staroperad^{\rho\times\psi}(B_V\times B_W)\] which we will denote by $\phi$ throughout this section. 
Our goal for this section is to prove the following.

\begin{theorem}\label{thm:basicadditivity}
    The induced map \[\phi: \staroperad^\rho(B_V)\otimes \staroperad^\psi(B_W)\to \staroperad^{\rho\times\psi}(B_V\times B_W)\] is a weak equivalence of $G$-operads.
\end{theorem}

In order to prove that $\phi$ is a weak equivalence, we will first show that there is a sub-collection $\additivecore \subseteq \staroperad^{\rho}(B_V)\otimes\staroperad^{\psi}(B_W)$ such that $\phi$ restricted onto $\mathcal{K}$ is an embedding.
We will then show that there exists a weak deformation retract of $\staroperad^{\rho\times\psi}(B_V\times B_W)$ onto $\additivecore$ \emph{which lifts to a weak deformation retract of $\staroperad^\rho(B_V)\otimes \staroperad^\psi(B_W)$ onto $\additivecore$}, thus proving \Cref{thm:basicadditivity}.

\subsection{Injectivity on the Additive Core}
Given any affine dilation map $x\in \DA(V)$, the image $x(B_V)$ is an open ball, and we will write $\Rad{x;B_V}$ for its radius.
The following is a critical geometric lemma.

\begin{lemma}\label{lem:criticalgeometricdisks}
    Suppose $x,y_1,y_2\in \DA(V)$ and $\lambda > 1$ where \begin{enumerate}
        \item $x(B_V)\cap y_i(\vB[V])\neq \emptyset$ for each $i$, and
        \item $y_1(\lambda \vB[V])\cap y_2(\lambda \vB[V]) = \emptyset$.
    \end{enumerate}
    Then \[\Rad{x;B_V} > \frac{\lambda - 1}{2}\bigg(\Rad{y_1;B_V}+\Rad{y_2;\vB[V]}\bigg).
    \] Moreover, for any $\mu \geq \frac{4}{\lambda - 1}+3$ we have for each $i$, \[y_i(\lambda B_V) \subseteq x(\mu B_V).\]
\end{lemma}
\begin{proof}
    Fix the radii of $y_1(B_V)$ and $y_2(B_V)$.
    The furthest the two disks $y_1(B_V)$ and $y_2(B_V)$ can be and still intersect $x(B_V)$ is (in the limit) if they intersect it tangentially at antipode points.
    Thus, it follows that the smallest that the radius of $x(B_V)$ can be while still keeping $y_1(\lambda B_V)\cap y_2(\lambda B_V) = \emptyset$ is if the diameter of $x(B_V)$ is more than $$(\lambda-1)\rad(y_1) + (\lambda -1)\rad(y_2).
    $$ See \cref{lem:criticalgeometricdisks figure 1} for a picture of this situation.

    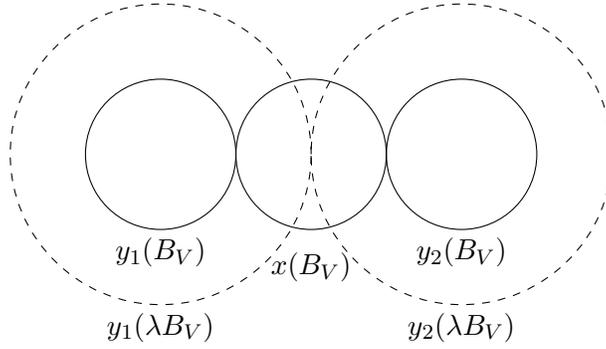
\begin{figure}[H]
        \centering
        \begin{tikzpicture}
            \draw (0,0) circle (1);
            \draw[dashed] (0,0) circle (2);
            \draw (2,0) circle (1);
            \draw (4,0) circle (1);
            \draw[dashed] (4,0) circle (2);
            \node at (0,-1.3) {$y_1(B_V)$};
            \node at (0,-2.3) {$y_1(\lambda B_V)$};
            \node at (2,-1.5) {$x(B_V)$};
            \node at (4,-1.3) {$y_2(B_V)$};
            \node at (4,-2.3) {$y_2(\lambda B_V)$};
        \end{tikzpicture}
        \caption{The extreme case}
        \label{lem:criticalgeometricdisks figure 1}
    \end{figure}

    Hence, we deduce that \[\Rad{x;B_V} > \frac{\lambda - 1}{2}\bigg(\Rad{y_1;B_V}+\Rad{y_2;\vB[V]}\bigg).
    \] The second statement follows as again, we can think about the case when $y_i(B_V)$ are the furthest they can be apart.
    In order for $x(\mu B_V)$ to cover $y_1(\lambda B_V)\cup y_2(\lambda B_V)$, we need that for each $i$ that \[(\mu-1) \Rad{x;B_V} \geq (\lambda + 1)\Rad{y_i;B_V}.
    \] Since we have that $\Rad{x;B_V} > \frac{\lambda -1}{2}\Rad{y_i;B_V}$, we get that \begin{align*}
        (\mu -1)\Rad{x;B_V} > \frac{(\mu -1)(\lambda - 1)}{2}\Rad{y_i;B_V}
    \end{align*}
    Hence, the condition is met as long as \begin{align*}
        \frac{(\mu -1)(\lambda - 1)}{2} & \geq \lambda + 1                                                                                            \\
        \mu - 1                         & \geq \frac{2(\lambda +1)}{\lambda - 1}                                                                      \\
                                        & = \frac{2(\lambda - 1 + 2)}{\lambda -1}                                                                     \\
                                        & = 2+\frac{4}{\lambda - 1} \shortintertext{and so we get $y_i(\lambda B_V) \subseteq x(\mu B_V)$ as long as}
        \mu                             & \geq 3 + \frac{4}{\lambda -1}
    \end{align*}
\end{proof}

Note that if we take $\lambda=\mu=5$ then the second condition works.
This leads to the following definition which we will show has useful division properties as a consequence of \cref{lem:criticalgeometricdisks}.

\begin{definition}{}{}
    The \emph{separated $\rho$-framed little disk operad} $\separateddiskoperad^\rho(B_V)$ is a suboperad of $\staroperad^\rho(B_V)$ given by the following \[\separateddiskoperad^\rho(B_V):= \set[\bigg]{x\in \staroperad^\rho(B_V)\given x\circ(5\id)_{i\in\ar(x)}\in \staroperad^\rho(B_V)\text{ if } \abs{\ar(x)}>1}\]
    Here we interpret $\staroperad^\rho(B_V) \subseteq \ambientoperad^\rho$ so the composition $x\circ(5\id)_{i\in\ar(x)}$ happens in $\ambientoperad^\rho$.
\end{definition}

\begin{remark}
    Geometrically, the operad $\separateddiskoperad^\rho(B_V)$ can be described as tuples $x=(x_i)_{i\in\ar{x}}$ where the disks $x_i(B_V)$ are such that if we enlarge them $5$ times, then they don't overlap, and they are still contained within $B_V$. See \cref{fig:seperated disks} for a picture of this. 
\end{remark}

\begin{figure}[H]
    \centering
    \begin{tikzpicture}
        \draw (0,0) circle (3);
        \draw (-1.5,0) circle (0.5);
        \draw[dashed] (-1.5,0) circle (1.5);
        \draw (1,1) circle (0.25);
        \draw[dashed] (1,1) circle (1);
        \draw (1.5,-1) circle (0.25);
        \draw[dashed] (1.5,-1) circle (1);
    \end{tikzpicture}
    \caption{An element of $\separateddiskoperad^\rho(B_V)$.}
    \label{fig:seperated disks}
\end{figure}

Recall the intersection correspondence of \cref{def:intersectioncorrespondence}.
Given $x,y\in \staroperad^\rho(B_V)$, we will define the following sets which partition the indexing sets of $x$ and $y$.
\begin{align*}
    L^1_{x,y} & := \set{i\in \ar(x)\given \text{ if } i\cap_{x,y}j\text{ then } \rad(x_i)\leq \rad(y_j)} \\
    L^2_{x,y} & := \set{j\in \ar(y)\given \text{ if } i\cap_{x,y}j\text{ then } \rad(x_i)> \rad(y_j)}    \\
    R^1_{x,y} & := \ar(x)\backslash L_{x,y}^1                                                            \\
    R^2_{x,y} & := \ar(y)\backslash L_{x,y}^2.
\end{align*}

\begin{example}
    As an example of what these sets correspond to, consider \cref{fig:seperatedaritysets}. 

    \begin{figure}[H]
        \centering
        \begin{tikzpicture}
            \draw (0,0) circle (3);
            \draw (-1.5,0) circle (1) node {$x_1$};
            \draw (1,1) circle (1) node {$x_2$};
            \draw[dashed] (1.5,-1) circle (0.5) node {$y_2$};
            \draw[dashed] (-0.5,1) circle (1) node {$y_1$};
            \draw (0.8,-1.8) circle (0.7) node {$x_3$};
        \end{tikzpicture}
        \caption{Example for indexing partitions}
        \label{fig:seperatedaritysets}
    \end{figure}
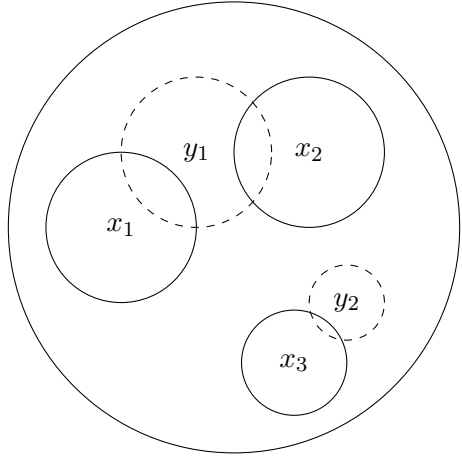

    Here, we have that \begin{align*}
        \vL[x,y,upper=1] &= \set{1,2} \\ 
        \vL[x,y,upper=2] &= \set{2} \\
        \vR[x,y,upper=1] &= \set{3} \\
        \vR[x,y,upper=2] &= \set{1}. 
    \end{align*} 
    The ``$L$'' sets correspond to indexes of disks which are smaller than every disk it intersects, while ``$R$'' sets are the complement. There is an asymmetry in the definition to ensure that there is a preferred choice for when disks of equal radius intersect. This is needed for \cref{lem:criticalgeometricdiskstwo}.
\end{example}

\begin{lemma}\label{lem:criticalgeometricdiskstwo}
    Suppose $x,y\in \separateddiskoperad^\rho(B_V)$ and under the intersection correspondence $\cap_{x,y}$ every element of $\ar(x)$ corresponds to some element in $\ar(y)$ and vice versa.
    Then we have the following
    \begin{enumerate}
        \item If $\abs{c_{x,y}(i)}>1$ then $i\in \vR[x,y,upper=1]$ and for all $j\in c_{x,y}(i)$ we have that $j\in \vL[x,y,upper=2]$ and \[y_j(5B_V) \subseteq x_i(5B_V).\] Moreover, we have that
        \item $c_{x,y}(L_{x,y}^1)=R_{x,y}^2$, and $c_{x,y}(R_{x,y}^1)=L_{x,y}^2$.
    \end{enumerate}
    Similar statements hold for $c_{y,x}$.
\end{lemma}

\begin{proof}
    The first statement follows from the second part of \cref{lem:criticalgeometricdisks}, while the second statement follows from the first part of \cref{lem:criticalgeometricdisks}.
    In particular, if $\abs{c_{x,y}(i)}>1$ then for all $j\in c_{x,y}(i)$ we have \[\Rad{x_i,B_V}>2\Rad{y_j,B_V},\] and we conclude that $c_{y,x}(j)=\set{i}$.
\end{proof}

We will use this lemma to construct certain elements in $\staroperad^\rho(B_V)$ from $x,y\in \separateddiskoperad^\rho(B_V)$.
In particular, we will define elements $x\triangleleft y, x \triangleright y, x \triangledown y \in \staroperad^\rho(B_V)$ via the following.
\begin{align*}
    (x\triangleright y)_i & := \begin{cases*}
                                   x_i & if  $i\in L_{x,y}^1$ \\ x_i\circ 5\id & if  $i\in R_{x,y}^1$
                               \end{cases*} \\
    (x\triangleleft y)_j  & := \begin{cases*}
                                   y_j & if  $j\in L_{x,y}^2$ \\ y_j\circ 5\id & if  $j\in R_{x,y}^2$
                               \end{cases*} \\
    (x\triangledown y)_i  & := \begin{cases*}
                                   x_i & if  $i\in L_{x,y}^1$ \\ y_i & if  $i\in L_{x,y}^2$
                               \end{cases*}
\end{align*}
Here, $\ar(x\triangleright y)=\ar(x), \ar(x\triangleleft y)=\ar(y)$, and $\ar(x\triangledown y)=\vL[x,y]^1\cup\vL[x,y]^2$ which we put the induced order on.

\begin{example}
    Let us geometrically describe what these constructed elements are doing. For this example consider \cref{fig bubbleexample}. In \cref{fig:bubble example 1} we have two examples of \(x,y\in \separateddiskoperad^\rho(B_V)\) and for the disks that are larger than those they intersect, we have drawn in dotted lines the disk that is enlarged by $5$ times. Such an enlarged disk covers the disks that intersect the original disk. The element $x\triangleright y$ is then shown in \cref{fig:bubble example 2}. This is just the element $x$ where we have replaced those disks with the enlarged version. The element $x\triangleleft y$ is similar but for the element $y$ and is shown in \cref{fig:bubble example 4}. The element $x\triangledown y$ is shown in \cref{fig:bubble example 3} and this is the disks of $x$ and $y$ which are the smaller disks in each intersection. The key observation about these elements is that we have a zigzag of inclusions of the disks as follows.
    
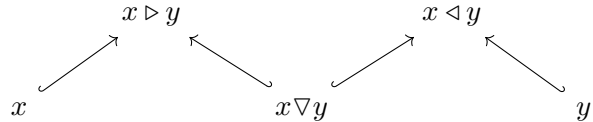
\begin{center}
        \begin{tikzcd}
            &x\triangleright y \ar[dr,hookleftarrow]&&x\triangleleft y\ar[dr,hookleftarrow]& \\ 
            x\ar[ur,hookrightarrow]&&x\triangledown y\ar[ur,hookrightarrow]&&y
        \end{tikzcd}
    \end{center}
    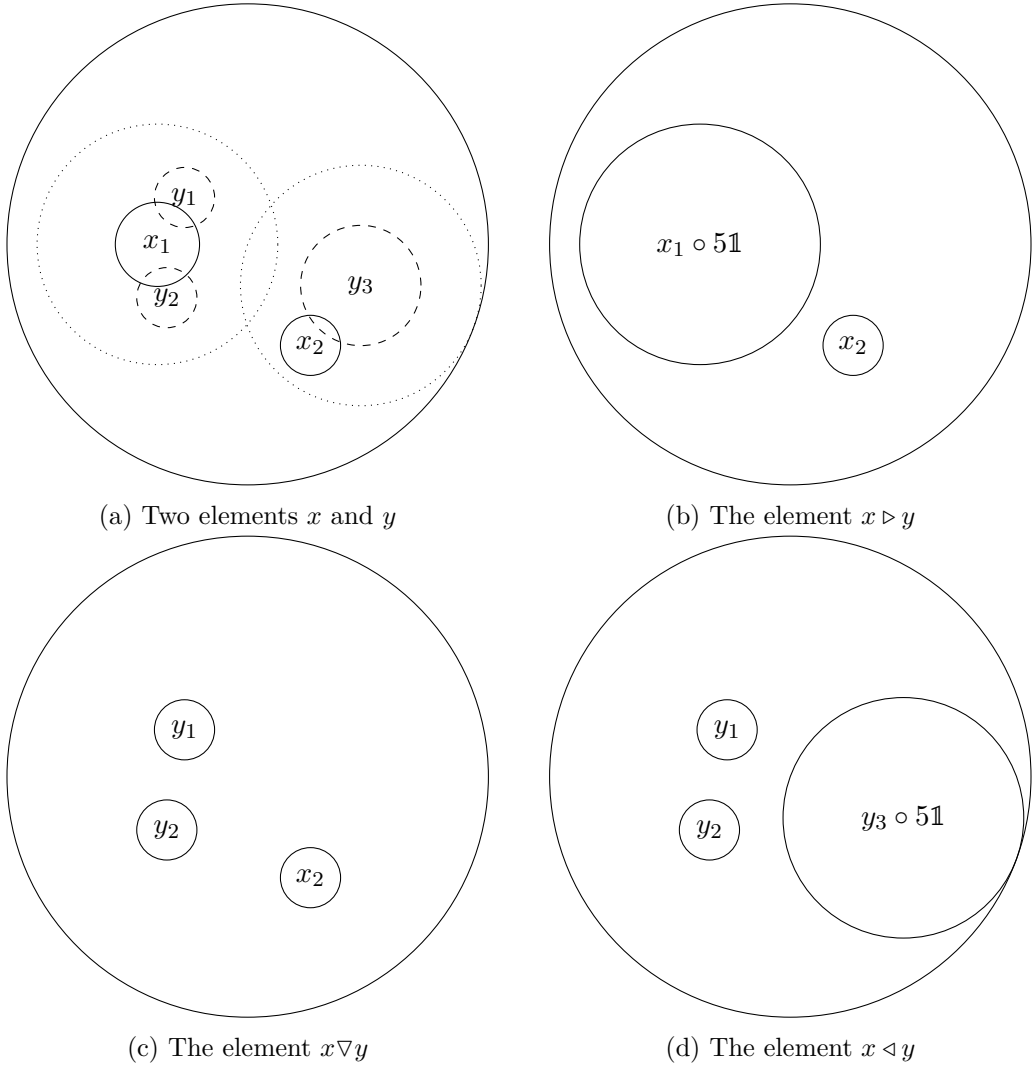
\begin{figure}[ht]
        \centering
        \begin{subfigure}{0.45\textwidth}
            \centering
            \begin{tikzpicture}[scale=0.8]
                \draw (0,0) circle (4);
                \draw (-1.5,0) circle (0.7) node {$x_1$};
                \draw[dotted] (-1.5,0) circle (2);
                \draw[dashed] (-1.5,0) ++(60:0.9) circle (0.5) node {$y_1$};
                \draw[dashed] (-1.5,0) ++(-80:0.9) circle (0.5) node {$y_2$};
                \draw[dashed] (-20:2) circle (1) node {$y_3$};
                \draw[dotted] (-20:2) circle (2);
                \draw (-20:2) ++ (230:1.3)circle (0.5) node {$x_2$};
            \end{tikzpicture}
            \caption{Two elements $x$ and $y$}
            \label{fig:bubble example 1}
        \end{subfigure}
        \begin{subfigure}{0.45\textwidth}
            \centering
        \begin{tikzpicture}[scale=0.8]
            \draw (0,0) circle (4);
            \draw (-1.5,0) circle (2) node {$x_1\circ 5\id$};
            \draw (-20:2) ++ (230:1.3)circle (0.5) node {$x_2$};
        \end{tikzpicture}
        \caption{The element $x\triangleright y$}
        \label{fig:bubble example 2}
    \end{subfigure}
        \begin{subfigure}{0.45\textwidth}
            \centering
        \begin{tikzpicture}[scale=0.8]
            \draw (0,0) circle (4);
            \draw (-1.5,0) ++(60:0.9) circle (0.5) node {$y_1$};
            \draw (-1.5,0) ++(-80:0.9) circle (0.5) node {$y_2$};
            \draw (-20:2) ++ (230:1.3)circle (0.5) node {$x_2$};
        \end{tikzpicture}
        \caption{The element $x\triangledown y$}
        \label{fig:bubble example 3}
    \end{subfigure}
    \begin{subfigure}{0.45\textwidth}
        \centering
        \begin{tikzpicture}[scale=0.8]
            \draw (0,0) circle (4);
            \draw (-1.5,0) ++(60:0.9) circle (0.5) node {$y_1$};
            \draw (-1.5,0) ++(-80:0.9) circle (0.5) node {$y_2$};
            \draw (-20:2) circle (2) node {$y_3\circ 5\id$};
        \end{tikzpicture}
        \caption{The element $x\triangleleft y$}
        \label{fig:bubble example 4}
    \end{subfigure}
    \caption{Example of constructed elements}
    \label{fig bubbleexample}
    \end{figure}
\end{example}

We can make our observations from the example rigorous by using \cref{lem:criticalgeometricdiskstwo}. Using the connection between the images of the disks and divisibility we can see that for each $i\in R_{x,y}^1$, there exists $\mu^i\in \separateddiskoperad^\rho(B_V)(c_{x,y}(i))$ such that for each $j\in c_{x,y}(i)$ we have that \[y_j=x_i\circ 5\id\circ \mu^i_j.
\] Hence, if we also set $\mu^i = \id$ for $i\in L_{x,y}^1$, we get that \[(x\triangleright y)\circ (\mu^i)_{i\in \ar(x)}= x\triangledown y,\] and by setting \[\overline{\mu}^i=\begin{cases*}\id_V& if $i\in L_{x,y}^1$ \\ \frac{1}{5}\id_V &if $i\in R_{x,y}^1$ \end{cases*}\] we have \[x = (x \triangleright y)\circ (\overline{\mu}^i)_{i\in\ar(x)}.\]We similarly have that for each $j\in \ar(y)$ there exists $\nu^j\in \separateddiskoperad^\rho(B_V)(c_{y,x}(j))$ and $\overline{\nu}^j\in\separateddiskoperad^\rho(B_V)(\set{j})$ such that \[(x\triangleleft y)\circ (\nu^j)_{j\in \ar(y)}= x\triangledown y,\] and \[y = (x\triangleleft y)\circ (\overline{\nu}^j)_{j\in \ar(y)}\]

Summarizing this discussion, we get the following lemma.

\begin{lemma}\label{lem:bubbletransfer}
    Given $x,y\in \separateddiskoperad^\rho(B_V)$ such that for each $i\in\ar(x)$, there exists $j\in \ar(y)$ such that $i\cap_{x,y} j$ and vice versa.
    There exists:
    \begin{enumerate}
        \item $x\triangleleft y, x \triangleright y, x \triangledown y \in \staroperad^\rho(B_V)$,
        \item for each $i\in \ar(x)$, elements $\mu^i\in\separateddiskoperad(V)(c_{x,y}(i))$, $\overline{\mu}^i\in\separateddiskoperad^\rho(B_V)(\set{i})$, and
        \item for each $j\in \ar(y)$, elements $\nu^j\in\separateddiskoperad^\rho(B_V)(c_{y,x}(j))$, $\overline{\nu}^j\in\separateddiskoperad^\rho(B_V)(\set{j})$,
        \item $\sigma_x,\sigma_y\in \Aut(\ar(x \triangledown y))$,
    \end{enumerate}
    such that the following hold
    \begin{align*}
        x                & = (x \triangleright y)\circ (\overline{\mu}^i)_{i\in\ar(x)}, \\
        y                & = (x\triangleleft y)\circ (\overline{\nu}^j)_{j\in \ar(y)},  \\
        x\triangledown y & = \sigma_x\cdot\pbr{(x\triangleright y)\circ (\mu^i)_{i\in \ar(x)}}            \\ &= \sigma_y\cdot \pbr{(x\triangleleft y)\circ (\nu^j)_{j\in \ar(y)}}.
    \end{align*}
\end{lemma}

\begin{definition}{}{}
    The \emph{additive core} $\additivecore$ is the image of $\coregenerators(\separateddiskoperad^\rho(B_V),\separateddiskoperad^\psi(B_W))$ under the quotient map \[q: \supergenerators(\staroperad^\rho(B_V),\staroperad^\psi(B_W))\to \staroperad^\rho(B_V)\otimes \staroperad^\psi(B_W).
    \] This is a $G$-collection.
\end{definition}

A bit more concretely, elements of $\additivecore$ are those that have representations of the form \[\big(\sigma\cdot(a\otimes b)\big)\circ \big(c^{i,j}\otimes d^{i,j}\big)_{(i,j)\in\ar{a}\times \ar{b}}\] where \begin{enumerate}
    \item $a\in\separateddiskoperad^\rho(B_V)$, $b\in\separateddiskoperad^\psi(B_W)$ and $\sigma\in \Sigma_{\ar{a\otimes b}}$, 
    \item $c^{i,j}\in \staroperad^\psi(B_V)_{\leq 1}$ and $d^{i,j}\in\staroperad^\psi(B_W)_{\leq 1}$,  
    \item $\abs{\ar{c^{i,j}}}=0$ if and only if $\abs{\ar{d^{i,j}}}=0$, and 
    \item If $\abs{\ar{c^{i,j}}}=0$, then there exists a $\vj[']$ such that $\abs{\ar{d^{i,\vj[']}}}\neq 0$ and vice versa.
\end{enumerate}

\begin{theorem}\label{theorem: injective on core}
    The induced map $\phi:\staroperad^\rho(B_V)\otimes \staroperad^\psi(B_W)\to \staroperad^{\rho\times\psi}(B_V\times B_W)$ restricted to $\additivecore$ is an embedding.
\end{theorem}

\begin{proof}
    Since $\coregenerators(\separateddiskoperad^\rho(B_V),\separateddiskoperad^\psi(B_W))$ is closed in $\propergenerators(\staroperad^\rho(B_V),\staroperad^\psi(B_W))$, \cref{lemma on induced map are closed} implies that if the induced map $\phi$ is injective on $\additivecore$, then it's an embedding.
    We prove that this map is injective via induction. From \cref{lem:unaryinjection}, we know that $\phi_1$ is injective.
    Suppose that $\phi$ is injective on $\additivecore_k$ for all $k<n$, and we have $x,y\in \additivecore_n$ such that $\phi(x)=\phi(y)$.

    A consequence of the definition of $\additivecore$ is that we have representations of $x,y$ in the form:

    \begin{align*}
        x & = (a\otimes b)\circ \left(c^{(i,j)}\otimes d^{(i,j)}\right)_{(i,j)\in \ar(a)\times\ar(b)} \\
        y & = (e\otimes f)\circ \left(g^{(s,t)}\otimes h^{(s,t)}\right)_{(s,t)\in \ar(e)\times\ar(f)}
    \end{align*}
    where for $x$ we have:
    \begin{enumerate}
        \item $\abs{\ar(a\otimes b)}>1
              $, and $\abs{\ar(c^{(i,j)})}=\abs{\ar(d^{(i,j)})}=0,$ or $1$ simultaneously for each pair $(i,j)$,
        \item for all $i$, there exists $j$ such that $\abs{\ar(c^{(i,j)})}=\abs{\ar(d^{(i,j)})}=1$ and similarly, for all $j$ there exists an $i$ with $\abs{\ar(c^{(i,j)})}=\abs{\ar(d^{(i,j)})}=1$.
    \end{enumerate}
    The representation for $y$ has the analogous properties.

    \emph{Case 1: $\abs{\ar(b)}=\abs{\ar(e)}=1$.}
    Observe that in this case that we have (after a chosen isomorphism) $\ar(a)\times\ar(b)=\ar(a)$, and so by the above conditions, none of the elements $c^i,d^i$ are nullary.
    Hence, as $c^i$ are unary, they interchange, and we have that \begin{align*}
        x & = (a\otimes b)\circ \left(c^{i}\otimes d^{i}\right)_{i\in \ar(a)}                     \\
          & = \biggl(a\circ (c^i)_{i\in\ar(a)}\biggr)\circ \biggl(b\circ d^i\biggr)_{i\in\ar(a)}.
    \end{align*}
    Similarly, we have that \begin{align*}
        y & = \biggl(f\circ (h^t)_{t\in\ar(f)}\biggr)\circ \biggl(e\circ g^t\biggr)_{t\in\ar(f)}.
    \end{align*}
    We can take that $\ar(a)=\ar(f)=\ar(x)=\ar(y)$, and as $\phi(x)=\phi(y)$, we have that the projections are equal.
    In particular, we deduce that if we set \begin{align*}
        A & := a\circ (c^i)_{i\in\ar(a)}  \\
        F & := f\circ (h^t)_{t\in\ar(f)},
    \end{align*}
    then we have that for each $i\in \ar(A)=\ar(F)$, that\begin{align*}
        A_i & = a_i\circ c^i = e\circ g^i \\
        F_i & = f_i\circ h^i = b\circ d^i
    \end{align*} in $\mathcal{G}^\rho$ and $\mathcal{H}^\psi$ respectively.

    Consequently, we then get that \begin{align*}
        x & = \biggl(a\circ (c^i)_{i\in\ar(a)}\biggr)\circ \biggl(b\circ d^i\biggr)_{i\in\ar(a)} \\
          & = A \circ (F_i)_{i\in\ar(A)}                                                         \\
          & = A \circ \biggl(F\circ (\delta^i_j)_{j\in \ar(F)}\biggr)_{i\in\ar(A)}               \\
          & = (A\otimes F)\circ (\delta^i_j)_{j\in \ar(F)}                                       \\
          & = F \circ \biggl(A\circ (\delta^i_j)_{j\in \ar(A)}\biggr)_{i\in\ar(F)}               \\ &= F \circ (A_i)_{i\in\ar(F)} \\ &= y.
    \end{align*}
    \emph{Case 2: $\abs{\ar(a)},\abs{\ar(e)}>1$.}

    In this case, since $a,e\in \separateddiskoperad^\rho(B_V)_{>1}$,  \cref{lem:bubbletransfer} tells us we have the following elements and equations:
    \begin{align}
        a                & = (a\triangleright e)\circ (\overline{\mu}^i)_{i\in\ar(a)} \label{lem:injective eq 1}\\
        e                & = (a\triangleleft e)\circ (\overline{\nu}^s)_{s\in\ar(e)}  \label{lem:injective eq 2}\\
        a\triangledown e & =  (a\triangleright e)\circ (\mu^i)_{i\in\ar(a)}           \label{lem:injective eq 3}\\
                         & = (a\triangleleft e)\circ (\nu^s)_{s\in\ar(e)}.  \label{lem:injective eq 4}
    \end{align}
    Writing \begin{align*}
        q^{i} & := b\circ (c^{(i,j)}\otimes d^{(i,j)})_{j\in \ar(b)} \shortintertext{and }
        r^s   & := f\circ (g^{(s,t)}\otimes h^{(s,t)})_{t\in \ar(b)},
    \end{align*} we set \[u^k := \begin{cases*}
            q^k & if $k\in L_{a,e}^1$ \\ r^k & if $k\in L_{a,e}^2$.
        \end{cases*}\]
    Since we have \begin{align*}
        \phi(x)&=\phi(y) \\ \phi(a\circ (q^i)_{i\in\ar(a)}) &= \phi(e\circ (r^s)_{s\in\ar(e)})
    \end{align*} by assumption, and from definitions we have for $k\in \vL[a,e,upper=1]$, and $\vk[']\in \vL[a,e,upper=2]$ that \begin{align*}
        \phi(a_k\circ q^k) &= \phi((a\triangledown e)_k\circ u^k) \\
        \phi(e_k\circ r^{\vk[']}) &= \phi((a\triangledown e)_{\vk[']}\circ u^{\vk[']}).
    \end{align*}
    Since $\staroperad^{\rho\times \psi}(B_V\times B_W)$ is algebraically axial, we deduce that \[\phi(a\circ (q^i)_{i\in\ar(a)}) = \phi((a\triangledown e)\circ(u^k)_{k\in L_{a,e}^1\cup L_{a,e}^2})=\phi(e\circ (r^s)_{s\in\ar(e)}).
    \] From the left equality, \cref{lem:injective eq 1,lem:injective eq 2,lem:injective eq 3,lem:injective eq 4}, and that $\staroperad^{\rho\times \psi}(B_V\times B_W)$ is left-cancellable, we get that for each $i\in R_{a,e}^1$ that \[\phi\left(\overline{\mu}^i\circ q^i\right) = \phi\left(\mu^i\circ (r^s)_{s\in c(i)}\right).\] Both of these inputs are elements of $\additivecore_{<n}
    $ and so by the inductive assumption we get that \[\overline{\mu}^i\circ q^i = \mu^i\circ (r^s)_{s\in c(i)}.
    \] Consequently, we deduce that \[x = (a\triangledown e)\circ(u^k)_{k\in L_{a,e}^1\cup L_{a,e}^2},\] and by symmetry we get $x=y$.
\end{proof}

\subsection{Weak Deformation Retracts onto the Additive Core}
We now switch focus and build the required weak deformation retracts.
In general, creating explicit homotopies on an operad tensor (even as collections) is difficult.
However, we can use a trick here.
All the homotopies that we have constructed come from composing in the operad and this works just as well in the tensor.
In particular, the following is a well-defined map of $G$-collections on $\staroperad^\rho(B_V)\otimes \staroperad^\psi(B_W)$:
\begin{align*}
    \widetilde{H}: \staroperad^\rho(B_V)\otimes \staroperad^\psi(B_W)\times [0,1) \to \staroperad^\rho(B_V)\otimes \staroperad^\psi(B_W) \\ \widetilde{H}(x,t) := x\circ \left(s_\rho(1-t)\otimes s_\psi(1-t)\right)_{i\in\ar(x)},
\end{align*}
where $s_\rho: [0,1)\to \vG[cal]\to\staroperad^\rho(B_V)$ is the section map from \cref{ass:topologicalgroups}.
This map is a lift of the following map
\begin{align*}
    H: \staroperad^{\rho\times \psi}(B_V\times B_W)\times [0,1) \to \staroperad^{\rho\times \psi}(B_V\times B_W) \\ H(x,t) := x\circ \left(s_{\rho\times \psi}(1-t)\right)_{i\in\ar(x)},
\end{align*}
and the following diagram commutes
$$
    \begin{tikzcd}
        \staroperad^\rho(B_V)\otimes \staroperad^\psi(B_W)\times [0,1) \arrow[r,"\widetilde{H}"]\arrow[d,"\phi\times \id"] & \staroperad^\rho(B_V)\otimes \staroperad^\psi(B_W)\arrow[d,"\phi"] \\ \staroperad^{\rho\times \psi}(B_V\times B_W)\times [0,1) \arrow[r,"H"]& \staroperad^{\rho\times \psi}(B_V\times B_W)
    \end{tikzcd}
$$
We would like to show that we can correct the maps $H$ and $\widetilde{H}$ to give weak deformation retracts onto $\additivecore$.
While it is fairly straightforward to show that we can do so for $H$, working with the tensor is difficult, and we will need to make an intermediate step. The issue is that while we can show that for every $x\in \staroperad^{\rho\times \psi}(B_V\times B_W)$ there is a large enough $t$ so that $H(x,t)\in \additivecore$. Since the induced map $\phi$ is not necessarily injective, it is non-trivial to show that for every element $x\in \staroperad^\rho(B_V)\otimes\staroperad^\psi(B_W)$, the deformation $\vH[~](x,t)$ eventually lies in $\additivecore$.

In the next section, we will prove the following lemma 
\begin{restatable}{lemma}{weakdeformationlemma}
    \label{lemma: weak deformation lemma}
    The equivariant map
        \begin{align*}
            \vH[~]: \staroperad^{\rho}(B_V)\otimes \staroperad^{\psi}(B_W)\times [0,1) \to \staroperad^{\rho}(B_V)\otimes \staroperad^{\psi}(B_W)\\ \vH[~](x,t) := x\circ \left(s_{\rho}(1-t)\id_V\otimes s_{\psi}(1-t)\id_W\right)_{i\in\ar(x)},
        \end{align*}
    is such that for each $x\in \staroperad^{\rho}(B_V)\otimes \staroperad^{\psi}(B_W)$, there exists an open neighborhood $V_x$ of $\phi(x)$, and $t_x\in (1-\frac{1}{50},1)$ such that for $U_x := \vphi[inv]V_x$, \[H(U_x,t_x)\subseteq \additivecore.\]
\end{restatable}

Assuming that this lemma holds, by using \cref{lem:weakdeformationcorrection}, there exists a $G$-invariant map $\lambda:\staroperad^{\rho\times\psi}(B_V\times B_W)\to [0,1)$ such that 
\begin{enumerate}
    \item for all $y\in \staroperad^{\rho\times \psi}(B_V\times B_W)$, we have $y\circ (\lambda(y))_{i\in\ar(y)}\in \additivecore$;
    \item for all $x\in \staroperad^{\rho}(B_V)\otimes \staroperad^\psi(B_W)$, we have $y\circ (\vlambda[~](x))_{i\in\ar(x)}\in \additivecore$ where $\vlambda[~]=\vlambda\circ\phi$.
\end{enumerate}
This then allows us to construct the weak deformation retracts by the following:
\begin{align*}
    \vD[~](x,t) &= x\circ \Big(\big(1-(1-\vlambda[~](x))t\big)(\id_V\otimes \id_W)\Big)_{i\in \ar{x}} \shortintertext{and} 
    \vD(x,t) &= x\circ \Big(\big(1-(1-\vlambda(x))t\big)(\id_{V\times W})\Big)_{i\in \ar{x}} 
\end{align*}
Thus proving \Cref{thm:basicadditivity}. Hence, all that is left is to prove \cref{lemma: weak deformation lemma}, which we do in the next section.

\subsection{Injectivity on Critical Elements}

In order to prove \cref{lemma: weak deformation lemma}, we will hone in on a special sub-symmetric sequence of the additive core $\additivecore$ that can be detected by the geometry of its image under $\phi$. 

An element $x\in \staroperad^{\rho\times \psi}(B_V\times B_W)$, has two projections $\operatorname{pr}_Vx\in \ambientoperad^\rho(B_V)$, and $\operatorname{pr}_Wx\in \ambientoperad^\psi(B_W)$ whose images can overlap.
As a result, the intersection relation gives two different partitions of $\ar(x)$:
\[\ar(x)=\coprod_{i\in I}P_i = \coprod_{j\in J}Q_j.
\] Here, $P_i$ is the intersection partition induced by the overlaps of $\operatorname{pr}_Vx$ and $Q_j$ induced by $\operatorname{pr}_Wx$.

\begin{definition}{}{}
    We will say $x\in \staroperad^{\rho\times\psi}(B_V\times B_W)$ is \emph{critical} if there exists $a\in \separateddiskoperad^\rho(B_V)(I)$ and $b\in \separateddiskoperad^\psi(B_W)(J)$ such that:
    \begin{enumerate}
        \item for all $i\in I$ we have that \[\emptyset\neq \bigcap_{k\in P_i}(\operatorname{pr}_Vx)_k(B_V) \subseteq \bigcup_{k\in P_i}(\operatorname{pr}_Vx)_k(B_V) \subseteq a_i(B_V),\] and
        \item for all $j\in J$ we have that \[\emptyset\neq \bigcap_{k\in Q_j}(\operatorname{pr}_Wx)_k(B_W) \subseteq \bigcup_{k\in Q_j}(\operatorname{pr}_Wx)_k(B_W) \subseteq b_j(B_W).\]
    \end{enumerate}
    We will denote the $G$-subcollection of $\staroperad^{\rho\times\psi}(B_V\times B_W)$ of all critical elements by $\mathscr{C}$, and also say that the pair of elements $(a,b)$ above \emph{separate} the critical element $x$.
\end{definition}

\begin{lemma}\label{lemma: exists core representation}
    For $w\in \vC[scr]$ which is separated by $(a,b)$, there exists unary or nullary pairs $(c^{(i,j)},d^{(i,j)})\in \left(\staroperad^\rho(B_V)\times \staroperad^\psi(B_W)\right)_{\leq 1}$ such that \begin{align*}
        w = \phi\left((a\otimes b)\circ (c^{(i,j)}\otimes d^{(i,j)})_{(i,j)\in\ar{a}\times \ar{b}}\right). 
    \end{align*} 
    Note that $(a\otimes b)\circ (c^{(i,j)}\otimes d^{(i,j)})_{(i,j)\in\ar{a}\times \ar{b}}\in\additivecore$.
\end{lemma}

\begin{proof}
    Since the images of the disks of $\phi(a\otimes b)$ cover those of $w$, there exists $r^{(i,j)}\in \staroperad^{\rho\times\psi}(B_V\times B_W)$ such that \begin{align*}
        w = \phi(a\otimes b)\otimes (r^{(i,j)})_{(i,j)\in\ar{a}\times\ar{b}}.
    \end{align*}
    We claim that $\abs{\ar{r^{(i,j)}}}=0,$ or $1$. Suppose for some $(\vi['],\vj['])\in \ar{a}\times\ar{b}$ that $\abs{\ar{r^{(\vi['],\vj['])}}}>1$. From the definition of $\vC[scr]$, we have that \begin{align*}
        \emptyset \neq \bigcap_{k\in \ar{r^{(\vi['],\vj['])}}} \pr_Vw_k(B_V) \shortintertext{and, } 
        \emptyset \neq \bigcap_{k\in \ar{r^{(\vi['],\vj['])}}} \pr_Ww_k(B_W). 
    \end{align*} This would imply that $w\notin \staroperad^{\rho\times\psi}(B_V\times B_W)$ which isn't the case. Hence, we know that $\abs{\ar{r^{(i,j)}}}=0,$ or $1$. Such elements we can write as \begin{align*}
        r^{(i,j)} = c^{(i,j)}\times d^{[i,j]}
    \end{align*} where $c^{(i,j)}\in S^\rho(B_V)$, $d^{(i,j)}\in S^\psi(B_W)$ and are either unary or nullary simultaneously. We can write \begin{align*}
        w &= \phi(a\otimes b)\circ (r^{(i,j)})_{(i,j)\in\ar{a}\times\ar{b}} \\ 
        &= \phi(a\otimes b)\circ (c^{(i,j)}\times d^{(i,j)})_{(i,j)\in\ar{a}\times\ar{b}} \\
        &= \phi\left((a\otimes b)\circ (c^{(i,j)}\otimes d^{(i,j)})_{(i,j)\in\ar{a}\times\ar{b}}\right).
    \end{align*}
    We are then done.
\end{proof}

\begin{definition}
    Given an element $z\in\vC[scr]$, \cref{lemma: exists core representation,theorem: injective on core} implies that there is a unique $\vz[~]\in \additivecore$ such that $\phi(\vz[~])=z$. We will call $\vz[~]$ the \emph{core representative} of $z$.
\end{definition}

\begin{definition}
    Define a $G$-subcollection of $\staroperad^\rho(B_V)\otimes \staroperad^\psi(B_W)$ by the following:
    \begin{align*}
        \vD[scr] := \set[\bigg]{x\circ \left(\frac{1}{50}\id_V\otimes \frac{1}{50}\id_W\right)\given x\in \staroperad^\rho(B_V)\otimes \staroperad^\psi(B_W)}.
    \end{align*}
    We will abuse notation slightly and for the corresponding $G$-subcollection of $\staroperad^\rho(B_V)$ we will write this as $\staroperad^\rho(B_V)\cap \vD[scr]$.
\end{definition}

Given a collection of elements $f^i\in \staroperad^\rho(B_V)$ for some $i\in \vI[cal]$.
We will say that this collection has \emph{transitive intersections} if for any $\vi[1],\vi[2],\vi[3]\in \vI[cal]$, and $\vk[1]\in\ar{f^{\vi[1]}}$,$\vk[2]\in\ar{f^{\vi[2]}}$,$\vk[3]\in\ar{f^{\vi[3]}}$ such that \[f^{\vi[1]}_{\vk[1]}(B_V)\cap f^{\vi[2]}_{\vk[2]}(B_V)\neq \emptyset, \text{ and } f^{\vi[2]}_{\vk[2]}(B_V)\cap f^{\vi[3]}_{\vk[3]}(B_V)\neq \emptyset \] then we have \[f^{\vi[1]}_{\vk[1]}(B_V)\cap f^{\vi[3]}_{\vk[3]}(B_V)\neq \emptyset.
\]

\begin{lemma}\label{lemma: common critical refinement}
    Let $\vI[cal]$ be an indexing set, and for each $i\in\vI[cal]$, suppose we have elements $e^i\in \staroperad^\rho(B_V)\cap \vD[scr]$ such that the collection $\set{e^i}_{i\in\vI[cal]}$ has transitive intersections.
    Then there exists $e\in\separateddiskoperad^\rho(B_V)$, $\vsigma[i]\in\Aut{\ar{e}}$, and $e^{i,k}\in \staroperad^\rho(B_V)_{\leq 1}$ such that \[e^i = (\sigma_i\cdot e)\circ (e^{i,k})_{k\in\ar{e}}.
    \]
\end{lemma}
\begin{proof}
    Looking at the collection of all the disks $\set{e^i_k}_{i\in\vI[cal],k\in\ar{e^i}}$, that the intersections are transitive implies we get an equivalence relation on this set via intersection.
    Denote these equivalence classes by $E_l$ for some $l\in\vL[cal]$, and we will put some linear order on $\vL[cal]$.
    Since each $e^i\in \staroperad^\rho(B_V)$, we have that for each $i\in\vI[cal]$ and $l\in \vL[cal]$ that $\set{e^i_k}_{k\in\ar{e^i}}\cap E_l$ is either empty or a singleton.
    For each $E_l$, pick a disk from it that has the largest diameter, say $e^l$.
    We then define $e\in \separateddiskoperad^\rho(B_V)(\vL[cal])$ by \[e_l:= e^l\circ 5\id_V.
    \] Note that each $e^i\in \vD[scr]$ ensures that $e\in\separateddiskoperad^\rho(B_V)$.
    Now, every element of $\set{e^i_k}_{i\in\vI[cal],k\in\ar{e^i}}$ is contained in exactly one $e_l$.
    The lemma then follows. 
\end{proof}

\begin{lemma}
    For any $z\in \vD[scr]$ such that $\phi(z)\in \vC[scr]$, let $\vz[~]$ be the core representation of $\phi(z)$.
    Then $z=\vz[~]$.
\end{lemma}
\begin{proof}
    We prove this via induction on the arity of $z$.
    Note that we know this is true on unary elements as $\phi$ is injective on this component via \cref{lem:unaryinjection}.
    Suppose $(a,b)$ separates $\phi(z)$.
    Then from \cref{lemma: exists core representation} we have that the core representative of $\phi(z)$ is given by \begin{align}
        \vz[~] := (a\otimes b)\circ (c^{(i,j)}\otimes d^{(i,j)})_{(i,j)\in\ar{a}\times \ar{b}} \label{crit lemma eq 1}
    \end{align} for some $c^{(i,j)}\in \staroperad^\rho(B_V)_{\leq 1}
    $, and $d^{(i,j)}\in \staroperad^\psi(B_W)_{\leq 1}$.

    Since $z\in\vD[scr]$, with-out loss of generality, there exists $\va[']\in \staroperad^\psi(B_V)_{>1}$ and $q^i\in \vD[scr]$ such that \begin{align}
        z = \va[']\circ (q^s)_{s\in\va[']}. \label{crit lemma eq 2}
    \end{align}
    As $\phi(\va[']\otimes b)$ covers $\phi(z)$, we deduce that there exists $e^{s,j}\in \staroperad^\rho(B_V)\cap\vD[scr]$ and $f^{\vs,\vj,\vk}\in\staroperad^\psi(B_W)\cap \vD[scr]$ such that \begin{align}
        \phi(z) = \phi\left((\va[']\otimes \vb)\circ \left(e^{s,j}\circ (f^{s,j,k})_{k\in\ar{e^{s,j}}}\right)_{(s,j)\in \ar{\va[']}\times \ar{\vb}}\right). \label{crit lemma eq 3}
    \end{align}
    Using \cref{crit lemma eq 2,crit lemma eq 3} and left cancellability, we deduce that for each $s\in\ar{\va[']}$ that \begin{align}
        \phi(q^s) = \phi\left(\vb\circ \left(e^{s,j}\circ (f^{s,j,k})_{k\in\ar{e^{s,j}}}\right)_{j\in \ar{\vb}}\right). \label{crit lemma eq 4}
    \end{align}
    From \cref{lemma: common critical refinement}, there exists $e^{s},e^{s,j,k}$ and $\sigma_j$ such that \[e^{s,j}= (\sigma_j\cdot e^s)\circ (e^{s,j,k})_{k\in \ar{e^s}}.\] 
    Denote the injective map on the arity sets by $\valpha[j]:\ar{e^{s,j}}\hookrightarrow \ar{e^s}$.
    We then calculate that
    \begin{align}
        \vb\circ \left(e^{s,j}\right)_{j\in \ar{\vb}} & = \vb\circ \left((\sigma_j\cdot e^s)\circ (e^{s,j,k})_{k\in \ar{e^s}}\right)_{j\in \ar{\vb}} \nonumber                \\
        & = \left(\vb\circ (\sigma_j\cdot e^s)_{j\in\ar{b}}\right)\circ (e^{s,j,k})_{(j,k)\in \ar{\vb}\times\ar{e^s}}\nonumber  \\
        \vb\circ \left(e^{s,j}\right)_{j\in \ar{\vb}} & =\left(\sigma\cdot (\vb\otimes e^s)\right)\circ (e^{s,j,k})_{(j,k)\in \ar{\vb}\times\ar{e^s}} \label{crit lemma eq 5}
    \end{align}
    where $\sigma = \oplus \sigma_j$.
    This allows us to extend \cref{crit lemma eq 4} to
    \begin{align}
        \phi(q^s) & = \phi\left(\vb\circ \left(e^{s,j}\circ (f^{s,j,k})_{k\in\ar{e^{s,j}}}\right)_{j\in \ar{\vb}}\right) \nonumber                                                                                       \\
                  & = \phi\left(\left(\vb\circ (e^{s,j})_{j\in\ar{b}}\right)\circ\left(f^{s,j,k}\right)_{(j,k)\in\ar{b}\times\ar{e^{s,j}}}\right) \nonumber                                                              \\
                  & = \phi\left(\left(\left(\sigma\cdot (\vb\otimes e^s)\right)\circ (e^{s,j,k})_{(j,k)\in \ar{\vb}\times\ar{e^s}}\right)\circ\left(f^{s,j,k}\right)_{(j,k)\in\ar{b}\times\ar{e^{s,j}}}\right) \nonumber \\
        \phi(q^s) & = \phi\left(\left(\sigma\cdot (\vb\otimes e^s)\right)\circ \left(e^{s,j,k}\otimes\vf[~]^{s,j,k}\right)_{(j,k)\in\ar{b}\times\ar{e^{s}}}\right) \label{crit lemma eq 6}
    \end{align}
    where $\vf[~]^{s,j,\vk[']} = \begin{cases*}
            \vf^{s,j,k} & if $\vk[']=\valpha[j](\vk)$ for some (unique) $\vk$\\ \ast & otherwise.
        \end{cases*}$

    Note that the element in the right side of \cref{crit lemma eq 6} is in $\additivecore$, so by the induction hypothesis we get that
    \begin{align}
        q^s & = \left(\sigma\cdot (\vb\otimes e^s)\right)\circ \left(e^{s,j,k}\otimes\vf[~]^{s,j,k}\right)_{(j,k)\in\ar{b}\times\ar{e^{s}}} \nonumber 
    \end{align}
    and so we conclude that 
    \begin{align}
        z   & = \va['] \circ \left(\left(\sigma\cdot (\vb\otimes e^s)\right)\circ \left(e^{s,j,k}\otimes\vf[~]^{s,j,k}\right)_{(j,k)\in\ar{b}\times\ar{e^{s}}} \right)_{s\in\ar{\va[']}} \nonumber                                                             \\
            & = \left(\va['] \circ \left(\sigma\cdot (\vb\otimes e^s)\right)_{s\in\ar{\va[']}}\right)\circ \left(e^{s,j,k}\otimes\vf[~]^{s,j,k}\right)_{(s,j,k)\in\ar{\va[']}\times\ar{b}\times\ar{e^{s}}} \nonumber                                           \\
        z   & = \left(\va['] \circ \left(\sigma\tau\cdot \left(e^s\circ (b)_{k\in\ar{e^s}}\right)\right)_{s\in\ar{\va[']}}\right)\circ \left(e^{s,j,k}\otimes\vf[~]^{s,j,k}\right)_{(s,j,k)\in\ar{\va[']}\times\ar{b}\times\ar{e^{s}}} \label{crit lemma eq 7}
    \end{align}
    where $\tau$ is the shuffle permutation. Writing $\widetilde{\sigma\tau}=\oplus_{s\in\ar{\va[']}}\sigma\tau$, we get that \begin{align}
        \va['] \circ \left(\sigma\tau\cdot \left(e^s\circ (b)_{k\in\ar{e^s}}\right)\right)_{s\in\ar{\va[']}} & = \widetilde{\sigma\tau}\cdot\left(\va[']\circ\left(e^s\circ (b)_{k\in \ar{e^s}}\right)_{s\in\ar{\va[']}}\right) \nonumber \\ & = \widetilde{\sigma\tau}\cdot\left(\left(\va[']\otimes e^s\right)\circ (b)_{(s,k)\in \ar{\va[']}\times\ar{e^s}} \right) \nonumber                                                                             \\ \va['] \circ \left(\sigma\tau\cdot \left(e^s\circ (b)_{k\in\ar{e^s}}\right)\right)_{s\in\ar{\va[']}} &= \widetilde{\sigma\tau}\cdot\left(\va[~]\otimes b\right) \label{crit lemma eq 8}
    \end{align}
    where $\va[~]= \va[']\circ (e^s)_{s\in\ar{\va[']}}
    $. Observe that as each $e^s\in \separateddiskoperad^\rho(B_V)$, we have that $\va[~]\in\separateddiskoperad^\rho(B_V)$.
    Putting this back into \cref{crit lemma eq 7} we get that \begin{align}
        z & = \left(\widetilde{\sigma\tau}\cdot (\va[~]\otimes b)\right)\circ \left(e^{s,j,k}\otimes\vf[~]^{s,j,k}\right)_{(s,j,k)\in\ar{\va[']}\times\ar{b}\times\ar{e^{s}}}
    \end{align} which is in $\additivecore$. Since $\phi$ is injective on $\additivecore$, we conclude that $z=\vz[~]$.
    Hence, we have proven the lemma.
\end{proof}

We can now prove the critical lemma of this section.

\weakdeformationlemma*
\begin{proof}
    For $x\in \staroperad^{\rho}(B_V)\otimes \staroperad^{\psi}(B_W)$, let $\lambda\in (0,1)$ be small enough that the element  \[z:=\vphi{x\circ (\lambda\id_V\otimes\id_W)_{i\in\ar{x}}}\] is such that if any of the projection disks $(\pr_Vz)_i$ or $(\pr_Wz)_i$ intersect then they share a common center. This implies that $z\in \vC[scr]$, and let us say that $z$ is separated by the pair $(a,b)$. Let $V_x \subseteq  \staroperad^{\rho\times \psi}(B_V\times B_W)(\ar{x})$ be a small open neighborhood of $x$ such that for any $y\in V_x$, the element $y\circ (\frac{\lambda}{2}\id_{V\times W})_{i\in \ar{x}}$ is in $\vC[scr]$ and is also separated by the pair $(a,b)$. It is not hard to see that this exists from geometric considerations.
    We then set $U_x:= \vphi[inv](V_x)$, and this is such that $H(U_x,1-\lambda/2)\in \vD[scr]\cap \vphi[inv]\vC[scr]$, and so we are done.
\end{proof}

\subsection{The General Case}

Let us now relax the condition that $\rho$ and $\psi$ are spherical representations. We have the following diagram 
\[\begin{tikzcd}
    \additivecore \ar[d,hookrightarrow]& \\ 
    \staroperad^{\rho_{tr}}(B_V)\otimes \staroperad^{\psi_{tr}}(B_W)\ar[r]\ar[d] & \staroperad^{\rho_{tr}\times \psi_{tr}}(B_V\times B_W) \ar[d,hookrightarrow]\\ 
    \staroperad^{\rho}(B_V)\otimes \staroperad^{\psi}(B_W)\ar[r] & \staroperad^{\rho\times \psi}(B_V\times B_W) \\ 
\end{tikzcd}\] where $\additivecore$ is the additive core of $\staroperad^{\rho_{tr}}(B_V)\otimes \staroperad^{\psi_{tr}}(B_W)$. Observe that we must have that $\additivecore$ embeds into $\staroperad^{\rho}(B_V)\otimes \staroperad^{\psi}(B_W)$. We can then use the homotopies of \Cref{spherical equivalence lemma}, which also lift onto the tensor, to show that the bottom row is a weak equivalence. We have then shown the following theorem.

\begin{theorem}
    Let $\rho$ and $\psi$ be dilation representations on the $G$-representations $V$ and $W$ respectively. Then the induced map of $G$-operads \[\staroperad^\rho(B_W)\otimes \staroperad^\psi(B_V)\to \staroperad^{\rho\times \psi}(B_W\times B_W)\] is a weak equivalence.
\end{theorem}

\begin{remark}
    While we have only proved this for a pair of dilation representations $\rho$, $\psi$ and open balls $B_V$, $B_W$. The techniques of this section work equally well for any number of finite representations and products of open balls.
\end{remark}

\appendix
\section{Homotopy Correction Lemmas}\label{appendix: correction lemmas}
Throughout this paper we often have spaces $A \subseteq X$ and maps $H:X\times [0,1)\to X$ such that for every $x\in X$, the flow line $t\mapsto H(x,t)$ eventually lands in $A$.  In this small appendix, we will give two elementary lemmas which give us sufficient conditions on when we can convert these into proper homotopies which exhibit the two spaces $A$ and $X$ as weak equivalences. 

For us, a weak $G$-retract is a $G$-map $r:X\to A$ such that if $i:A\to X$ is the inclusion, then $ri\simeq \id_A$ and $ir\simeq \id_X$. A weak $G$-deformation retraction is a $G$-homotopy $H:X\times [0,1]\to X$ such \begin{enumerate*}
    \item $H(x,0)=\id_X$,
    \item $H(X,1) \subseteq A$, and 
    \item we have a well-defined restriction $\restr{H}{A\times I}:A\times I\to A$.
\end{enumerate*} Such a homotopy exhibits $r(x):=H(x,1)$ as a weak $G$-retract.

\begin{lemma}\label{lem:weakdeformationcorrection}
    Given a $G$-space $X$ which is Hausdorff and paracompact, and a $G$-subspace $A \subseteq X$. If we have a continuous $G$-map $H:X\times [0,1) \to X$ such that for all $x\in X$, 
    \begin{conditions}
        \item\label{lem:weakdeformationcorrection condition 1} there exists an open neighborhood $U_x$ of $x$ such that for some $t_x\in [0,1)$ we have that $H(U_x,t_x) \subseteq A$,
        \item\label{lem:weakdeformationcorrection condition 2} if $H(x,t)\in A$, then $H(x,s)\in A$ for all $s\geq t$,
    \end{conditions}
    Then there exists a $G$-invariant map \[\lambda:X\to [0,1)\] such that the following homotopy \[D(x,t):=H(x,\lambda(x)t)\] is a weak $G$-deformation retract of $X$ onto $A$.
\end{lemma}

\begin{proof}
    We can build a cover of $G$-invariant subspaces $\set{G\cdot U_x}_x:= \set{\bigcup_{g\in G}gU_x}_x$, and as $X$ is paracompact and Hausdorff, take a partition of unity $\set{\overline{\rho}_i}_{i\in I}$ subordinate to $\set{G\cdot U_x}_x$. Averaging over the group \[\rho_i(x):= \frac{1}{\abs{G}}\sum_{g\in G}\vrho[i,oline](gx) \] then gives us a $G$-invariant partition of unity $\set{\vrho[i]}_{i\in I}$ subordinate to $\set{G\cdot U_x}$.
    For each $i\in I$, there exists an $x\in X$ such that \[\operatorname{supp}(\rho_i) \subseteq G\cdot U_x,\] and so for $t_i:= \max_{g\in G}\set{t_{gx}}$ from \cref{lem:weakdeformationcorrection condition 1,lem:weakdeformationcorrection condition 2}, we have that \[H(\operatorname{supp}(\rho_i),t_i)\subseteq A.\]
    For $x\in X$, we will write $I(x)=\set{i\in I \given \rho_i(x)\neq 0}$. 
   
   We construct a $G$-invariant function $\lambda: X \to [0,1)$ by \[\lambda(x) = \sum_{i\in I}t_i\rho_i(x). \] Note that we have that \begin{align*}
    \lambda(x) &= \sum_{i\in I}t_i\rho_i(x) \\ &= \sum_{i\in I(x)} t_i\rho_i(x) \\ &\geq \sum_{i\in I(x)}\min\set{t_i\given i\in I(x)}\rho_i(x) \\ &= \min\set{t_i\given i\in I(x)} \sum_{i\in I(x)}\rho_i(x) \\ &=  \min\set{t_i\given i\in I(x)} \end{align*}
    Hence, we have that $H(x,\lambda(x))\in A$ for all $x\in X$. It then follows that the map \begin{align*}
        D: X\times [0,1] &\to X \\
        D(x,t) &:= H(x,\lambda(x)t),
    \end{align*}
    is a weak $G$-deformation retract using \cref{lem:weakdeformationcorrection condition 2}.
\end{proof}

The following is a stronger result that gives sufficient conditions for a strong $G$-deformation retract instead. Note that while we prove this for general topological spaces using nets, in practice, this will only be applied to spaces in which the sequence characterization of continuity works.

\begin{lemma}\label{lem:deformationcorrection}
    Given a $G$-space $X$ and a closed $G$-subspace $A \subseteq X$. If we have a continuous $G$-map $H:X\times [0,1) \to X$ such that for all $x\in X$, 
    \begin{enumerate}
        \item there exist $t\in [0,1)$ such that $H(x,t)\in A$, 
        \item if $H(x,t)\in A$, then $H(x,s)\in A$ for all $s\geq t$,
        \item the set $\left(\set{x}\times [0,1)\right) \cap H^{-1}(\partial A)$ is a singleton.
    \end{enumerate}
    Then the mapping $\phi:x \mapsto t_x$ where $t_x= \min\set{t\in [0,1)\given H(x,t)\in A}$ is continuous and $G$-invariant. As a consequence, $A$ is a strong $G$-deformation retract of $X$ given by the deformation \[D(x,t) := H(x,\phi(x)t).\]
\end{lemma}

\begin{proof}
    Since $A$ is closed, the map \[\phi:X\to [0,1), \quad x \mapsto t_x\] is well-defined, and as $H$ is equivariant, it is $G$-invariant. So all we need to prove is that this map is continuous. We do this via the net characterization of continuity. Let $(x_i)_{i\in J}$ be a net in $X$ where $J$ is a directed set, and $x\in X$ such that $x_i\to x$ in $X$. We then need to prove that $\phi(x_i)\to \phi(x)$.

    If we had that $\phi(x_i)\nrightarrow \phi(x)$, then by viewing $\phi$ as having codomain $I=[0,1]$, which is compact, we must have that the net $(\phi(x_i))_{i\in J}$ has a cluster point which is \emph{not} $x$ in $I$. Hence, it is sufficient to show that every subnet of $(\phi(x_i))_{i\in J}$ that is convergent in $I$ must converge to $\phi(x)$.

    Suppose we have a subnet $(x_{f(s)})_{s\in S}$ of $(x_i)_{i\in I}$, where $f:S\to I$ is a monotone final function, such that $\phi(x_{f(s)})\to a\in I$.
    Let us consider the set $$C:=\set{s\in S \given \phi(x)=\phi(x_{f(s)})} \subseteq S.$$ If $C$ is cofinal in $S$, then we must have that $a=\phi(x)$ since $I$ is Hausdorff and limits are unique. Hence, let us assume that $C$ is not cofinal.
     
    As we have that \begin{align*}
        H(x_{f(s)},\phi(x_{f(s)}))\to H(x,a)
    \end{align*} and each $H(x_{f(s)},\phi(x_{f(s)}))\in A$, which is closed, we have that $H(x,a)\in A$ and so $0\leq \phi(x)\leq a$ by the definition of $\phi$. If $a=0$, then $\phi(x)=a$ in this case. Therefore, suppose we have that $a>0$. For $\epsilon>0$ consider the net $(y_s)_{s\in S} $ defined by \[y_s :=\max\left(0,\phi(x_{f(s)})-\epsilon\vert a-\phi(x_{f(s)})\vert \right)\] Since $C$ isn't cofinal, the net $(y_s)$ is such that eventually $y_s<\phi(x_{f(s)})$ and so $H(x_{f(s)},y_s)\in X\backslash A$. Since $y_s\to a$ it follows that $H(x,a)\in \partial A$ and $(x,a)\in\set{x}\times [0,1) \cap H^{-1}(\partial A)$. However, we also have that $(x,\phi(x))\in\set{x}\times [0,1) \cap H^{-1}(\partial A)$ and as this set is a singleton we must have $a=\phi(x)$.
\end{proof}

\section{Topology Lemmas}\label{appendix:topology lemmas}
In this appendix we prove some technical topology results we need in order to justify that our injective maps between operads are in fact embeddings. Recall \Cref{ass:topologicalgroups} and \Cref{filtered group conventions}.

We will first prove that the various suboperads are closed in the universal ambient operad.

\begin{lemma}\label{suboperads are closed lemma}
    For any star domain $S$ and dilation representation $\rho$, the operads $\staroperad^{\rho}(S)$ and $\ambientoperad^\rho(S)$ are closed in $\ambientoperad^\rho$.
\end{lemma}

\begin{proof}
    Let us first show that $\ambientoperad^\rho(S)$ is closed in $\ambientoperad^\rho$.
    Suppose we have a sequence $(x^k)_{k\in \mathbb{N}}\in \ambientoperad^\rho(S) $ such that $x^k\to x$ where $x\in \ambientoperad^\rho$.
    If $x\notin \ambientoperad^\rho(S)$, then there exists $p\in S$ such that for some $i\in\ar(x)$ we have that $x_i(p)\notin S$.
    Since the map $x_i$ is open, we can assume $p\notin\overline{S}$.
    However, $x_i^k(p) \to x_i(p)$ in $V$, but $x_i^k(p)\in S$ and so this is a contradiction.
    Hence, we deduce that $x\in \ambientoperad^\rho(S)$ and so $\ambientoperad^\rho(S)$ is closed in $\ambientoperad^\rho$.

    To prove $\staroperad^\rho(S)$ is closed we do something similar.
    Suppose we have a sequence $(x^k)_{k\in \mathbb{N}}\in \staroperad^\rho(S) $ such that $x^k\to x$ where $x\in \ambientoperad^\rho$.
    The previous argument shows that $x\in \ambientoperad^\rho(S)$, and so we need to show that the components of $x$ don't overlap.

    Suppose they do overlap.
    That is, for some $i\neq j\in \ar(x)$ we have that \[x_i(S)\cap x_j(S)\neq \emptyset.
    \]
    We recall that $\mathcal{G}^\rho(S)$ is a $G$-topological group and setting $y^k:=(x_j^k)^{-1}\circ x_i^k \to \id$, since the components of $x^k$ are disjoint, we have that \[y^k(S)\cap S = \emptyset.\]
    Since $y^k$ is a homeomorphism, we have that \[\overline{y^k(S)}\cap S = \emptyset.\]
    By our assumption, we have that \[y(S)\cap S \neq \emptyset\] this implies that there must be a point $p\in S$ such that $y^k(p)\to y(p)\in S$, but this is a contradiction.
    Therefore, $\staroperad^\rho(S)$ is closed in $\ambientoperad^\rho$.
\end{proof}

The following is an argument of \cite{barataAdditivityLittleCubes2022}.

\begin{lemma}\label{proper map from tensor lemma}
	Suppose we have interchanging operad maps \begin{align*}
		f & :P\to R \\ g&:Q\to R
	\end{align*} such that \begin{enumerate}
		\item the maps $f,$ and $g$ are proper, and
		\item the composition maps of $R$ are proper.
	\end{enumerate}
	Then the composition map \[\propergenerators(P,Q)\xrightarrow{q} P\otimes Q \xrightarrow{f\otimes g} R\] is proper.
\end{lemma}

\begin{proof}
	We have the following diagram of $G$-collections
	\[\begin{tikzcd}
			{\propergenerators(P,Q)} & {P\otimes Q} \\
			& R
			\arrow["f\otimes g", from=1-2, to=2-2]
			\arrow["q", two heads, from=1-1, to=1-2]
			\arrow["{(f\otimes g)\circ q}"', from=1-1, to=2-2]
		\end{tikzcd}\]
	Unpacking what the map $(f\otimes g)\circ q$ does on each component, it is a product of components of the $f$, and $g$ maps and then composes this with composition maps of $R$. For each $n$, $\propergenerators(P,Q)(n)$ has finite-many components. Since proper maps are closed under composition and finite products, the lemma then follows.
\end{proof}

Now, we want to apply this idea to our operads of interest. First, we note that one of the conditions \cref{proper map from tensor lemma} always holds for our operads.

\begin{lemma}\label{lemma dilation component includes is closed}
    Given dilation representations $\rho: \vG[cal]\to \DL(V_\bullet,W_\bullet)$, $\psi: \vH[cal]\to \DL(\vV[']_\bullet,\vW[']_\bullet)$, the operad map \[\vi[\rho]:\ambientoperad^\rho\to \ambientoperad^{\rho\times \psi}\] is proper. As a consequence of \cref{suboperads are closed lemma}, the corresponding maps for $\ambientoperad^\rho(S)$ and $\staroperad^\rho(S)$ are also proper.
\end{lemma}

\begin{proof}
    The operad map $\vi[\rho]$ has a left inverse given by the projection map \[\pr_\rho:\ambientoperad^{\rho\times \psi}\to \ambientoperad^\rho.\] So $\pr_\rho\circ \vi[\rho] = \id[\ambientoperad^\rho]$, and as $\id[\ambientoperad^\rho]$ is trivially proper, we have that $\vi[\rho]$ is proper.
\end{proof}

Using \cref{lemma dilation component includes is closed,proper map from tensor lemma}, if we write $\staroperad_i^{\rho}(S)$ for the little star operads from the groups $\vG[cal,i]$, we get that the maps \[\propergenerators(\staroperad_i^\rho(S),\staroperad_i^\psi(T))\to \staroperad_i^{\rho\times \psi}(S\times T)\] are proper. Taking the colimit, which is an increasing filtration, we get that the map \[\propergenerators(\staroperad^\rho(S),\staroperad^\psi(S))\to \staroperad^{\rho\times \psi}(S)\] is a closed map. We then have the following generalization of \cref{proper map from tensor lemma}.

\begin{lemma}\label{lemma on induced map are closed}
    For dilation representations $\rho$, and $\psi$ that satisfy \cref{filtered group conventions}, in the diagram \[\begin{tikzcd}
			{\propergenerators(\staroperad^\rho(S),\staroperad^\psi(T))} & {\staroperad^\rho(S)\otimes \staroperad^\psi(T)} \\
			& \staroperad^{\rho\times \psi}(S\times T)
			\arrow["{\vi[\rho]\otimes i_\psi}", from=1-2, to=2-2]
			\arrow["q", two heads, from=1-1, to=1-2]
			\arrow["{(i_\rho\otimes i_\psi)\circ q}"', from=1-1, to=2-2]
		\end{tikzcd}\]
        The maps $(i_\rho\otimes i_\psi)\circ q$ and $i_\rho\otimes i_\psi$ are closed.
\end{lemma}

This follows from the above and that $q$ is surjective on components. A consequence of this, is because the additive core $\additivecore$ is the image of a closed subspace of $\propergenerators(\staroperad^\rho(S),\staroperad^\psi(T))$, it follows that if $i_\rho\otimes i_\psi$ is injective on $\additivecore$, then this is an embedding.

We will now use the rest of this appendix to justify that \cref{filtered group conventions} hold for our two primary examples of interest. First, we observe the following.

\begin{lemma}{}{}
    If $V$ is a finite-dimensional $G$-representation, the multiplication map on the $\DA(V)$ is proper
\end{lemma}
\begin{proof}
    Recall that we have a homeomorphism \[\DA(V)\cong \orth(V)\times \mathbb{R}_{>0}\times V.\] For $0<\lambda_1<\lambda_2$ and $n\in \mathbb{N}$, we have the compact closed subspaces \[F_{\lambda_1,\lambda_2,n,n+1} := \set[\bigg]{(f,a,v)\in \orth(V)\times \mathbb{R}_{>0}\times V\given \lambda_1 \leq a \leq \lambda_2, n\leq \norm{v}\leq n+1}.\] Then $\set{F_{\lambda_1,\lambda_2,n,n+1}}$ is a closed cover and has a locally finite subcover. So it is sufficient to show that the multiplication map \[\mu: \DA(V)\times \DA(V)\to \DA(V)\] restricts to a proper map $\restr{\mu}{\mu^{-1}{F_{\lambda_1,\lambda_2,n}}}.$ This is easily seen by observing that there exists some $\lambda_1^{\prime},\lambda^{\prime}_2,$ and $M\in \mathbb{N}$ such that \[\mu^{-1}F_{\lambda_1,\lambda_2,n} \subseteq F_{\lambda_1^{\prime},\lambda^{\prime}_2,0,M}\times F_{\lambda_1^{\prime},\lambda^{\prime}_2,0,M}.\] Since this is a closed subspace of a compact Hausdorff space, we have the preimage is also compact. Hence, the lemma follows.
\end{proof}

We can then use this lemma to show the following 

\begin{lemma}\label{key convention verification lemma}
    Given a $G$-representation $V$ and $G$-filtration pair $(V_\bullet,V_\bullet^\prime)$. Let $A(V_i)$ be a closed $G$-subgroup of $\Orth(V_i)$. The group \[\vG[cal]=A(V_\bullet)\times \Lambda(V_i^\prime)\] where $A(V_\bullet)=\prod_i A(V_i)$ and the dilation representation given by the inclusion \[\vG[cal]\hookrightarrow \DL(V_\bullet,V_\bullet^\prime)\] satisfy the conditions of \cref{filtered group conventions}.
\end{lemma}

\begin{proof}
    Let $\set{0} \subseteq W_1 \subseteq W_2$ be a filtration of $V$ by $G$-representations. For each $i$, have the induced $G$-groups \[\vG[cal,i]^\rho \to \DA(V_\bullet\cap W_i,V_\bullet^\prime\cap W_i)\] which are closed subgroups of a product of the $\DA(W)$ groups with $W$ finite dimensional, so the multiplication maps are proper.
\end{proof}

Most importantly, the equivariant little disk operads $\littlediskoperad(V)$ and framed little disks $\littlediskoperad^{fr}_n$ are of the form in \cref{key convention verification lemma}.

\printbibliography
\addcontentsline{toc}{section}{Bibliography}

\end{document}